\documentclass[11pt]{article}
\usepackage{amsfonts}
\usepackage{lipsum}
\usepackage{latexsym}
\usepackage{amsmath}
\usepackage{amssymb}
\usepackage{amsthm}
\usepackage{graphicx}
\usepackage{fullpage}
\usepackage{enumerate}
\usepackage{color}
\usepackage{bbm}
\newtheorem{theorem}{Theorem}[section]
\newtheorem{lemma}[theorem]{Lemma}

\newtheorem{proposition}[theorem]{Proposition}

\newtheorem{remark}[theorem]{Remark}

\newtheorem{problem}[theorem]{Problem}
\theoremstyle{definition}

\newtheorem{thmy}{Theorem}

\newtheorem*{note*}{Note}
\newcommand{\measurerestr}{%
  \,\raisebox{-.127ex}{\reflectbox{\rotatebox[origin=br]{-90}{$\lnot$}}}\,%
}

\newcommand{\R}{\mathbb R}

\newcommand{\Sn}{\mathbb{ S}^{n-1}}

\newcommand{\Ha}{{{\cal H}^{n-1}}}

\makeatother

\newcommand{\proofbox}{\mbox{ $\Box$}\\}
\newcommand{\HH}{\mathcal{H}}
\newcommand{\KK}{\mathcal{K}}
\newcommand{\KKn}{\mathcal{K}^n}

\newcommand\blfootnote[1]{%
  \begingroup
  \renewcommand\thefootnote{}\footnote{#1}%
  \addtocounter{footnote}{-1}%
  \endgroup
}

\begin{document}


\title{\bf Uniqueness when the $L_p$ curvature is close to be a constant for $p\in[0,1)$}
\date{}
\medskip
\author{K\'{a}roly~J. B\"{o}r\"{o}czky, Christos Saroglou}
\maketitle
\begin{abstract}
For fixed positive integer $n$, $p\in[0,1)$, $a\in(0,1)$, we prove that if a function $g:\Sn\to \R$ is sufficiently close to 1, in the $C^a$ sense, then there exists a unique convex body $K$ whose $L_p$ curvature function equals $g$. This was previously established for $n=3$, $p=0$ by Chen, Feng, Liu \cite{CFL22} and in the symmetric case by Chen, Huang, Li, Liu \cite{CHLL20}. Related, we show that if $p=0$ and $n=4$ or $n\leq 3$ and $p\in[0,1)$, and the $L_p$ curvature function $g$ of a (sufficiently regular, containing the origin) convex body $K$ satisfies $\lambda^{-1}\leq g\leq \lambda$, for some $\lambda>1$, then $\max_{x\in\Sn}h_K(x)\leq C(p,\lambda)$, for some constant $C(p,\lambda)>0$ that depends only on $p$ and $\lambda$. This also extends a result from Chen, Feng, Liu \cite{CFL22}. Along the way, we 
obtain a result, that might be of independent interest, concerning the question of when the support of the $L_p$ surface area measure is lower dimensional. Finally, we establish a strong non-uniqueness result for the $L_p$-Minkowksi problem, for $-n<p<0$.
\end{abstract}

\noindent DATA AVAILABILITY STATEMENT required by CVPD: The paper uses no data.

\section{Introduction}

We call a compact convex set in $\R^n$ with non-empty interior a convex body. The family of convex bodies in $\R^n$ is denoted by $\KKn$, and we write $\KKn_o$ ($\KKn_{(o)}$) to denote the subfamily of $K\in\KKn$ with $o\in K$ ($o\in{\rm int}\,K$). We also write $\KKn_s$ to denote the class of symmetric convex bodies in $\R^n$. The support function of a  compact convex set $K$ is 
$h_K(u)=\max_{x\in K}\langle u,x\rangle$ for $u\in \R^n$, and hence $h_K$ is convex and $1$-homogeneous where a function $\varphi:\R^n\to\R$ is $\alpha$-homogeneous for $\alpha\in\R$ if
$\varphi(\lambda x)=\lambda^\alpha \varphi(x)$ for $\lambda>0$ and $x\in\R^n\backslash \{o\}$. The Euclidean unit ball in $\R^n$ centered at the origin is denoted by $B^n$ and its boundary, the Euclidean sphere, is denoted by $\Sn$. We write $\|x\|$ for the Euclidean norm of a vector $x\in\R^n$. Throughout this paper, we fix an orthonormal basis $\{e_1,\dots,e_n\}$ of $\R^n$. We will say that an operator $T:\R^n\to\R^n$ is diagonal if it is diagonal with respect to this basis. 

Let $\partial'K$ denote the subset of the boundary  of a $K\in\KKn$ such that there exists a unique exterior unit normal vector
$\nu_K(x)$ at any point $x\in \partial'K$.
It is well-known that $\HH^{n-1}(\partial K\setminus\partial'K)=0$ and $\partial'K$ is a Borel set  (see Schneider \cite{Sch14}) where $\HH^k$ is the $k$-dimensional Hausdorff measure normalized in a way such that it coincides with the Lebesgue measure on $k$ dimensional affine subspaces. The function $\nu_K:\partial'K\to \Sn$ is the spherical Gauss map
that is continuous on $\partial'K$.
 The surface area measure $S_K$ of $K$ is a Borel measure on $\Sn$ 
satisfying that $S_K(\eta)=\HH^{n-1}(\nu_K^{-1}(\eta))$
 for any Borel set $\eta\subset \Sn$. In particular, if $K\in\KKn$ and $\psi:\R\to\R$ is a measurable function that is bounded on any interval, then
\begin{equation}
\label{SKintegral}
\int_{\Sn}\psi\circ h_K\,dS_K=\int_{\partial' K}\psi(\langle x,\nu_K(x)\rangle)\,d\HH^{n-1}(x).
\end{equation}

If $S_K$ is absolutely continuous for a $K\in\KKn$, then we write $f_K$, the so-called curvature function, to denote the Radon-Nykodyn derivative satifying $dS_K=f_K\,d\HH^{n-1}$ (see Hug \cite{Hug96}, Ludwig \cite{Lud10} and Schneider \cite{Sch14}). 
We observe that $S_K$ is absolutely continuous if  there exists a constant $c>0$ such that 
$S_K<c\HH^{n-1}$ (meaning that $S_K(\omega)\leq c\HH^{n-1}(\omega)$ for any Borel set $\omega\subset \Sn$). For example, if 
$\partial K$ is $C^2_+$, then $f_K(u)$ is the reciprocal of the Gaussian curvature at $x\in\partial K$ where $u$ is the exterior unit normal for $u\in \Sn$, and hence 
\begin{equation}
\label{fKMonge}
\det(\nabla^2 h+h\,{\rm Id})= f_K
\end{equation}
where $h=h_K|_{\Sn}$, and $\nabla h$ and  $\nabla^2 h$ are the gradient and the Hessian of $h$ with respect to a moving orthonormal frame. 

For $K\in \KKn_o$ and $p\in\R$, Lutwak's $L_p$ surface area $S_{p,K}$ is defined by $dS_{p,K}=h^{1-p}dS_K$
for $h=h_K|_{\Sn}$
(see Lutwak \cite{Lut93}, B\"or\"oczky \cite{Bor23}); in particular, $S_{1,K}=S_K$ and $\frac1n\,S_{0,K}$ is the so-called cone-volume measure.
More precisely,  if $p>1$, then we need to assume that $h^{1-p}$ is $S_K$ integrable; for example, $K\in\KKn_{(o)}$. The corresponding $L_p$ Minkowski problem asks for the solution of the Monge-Amp\`ere equation
\begin{equation}
\label{LpMonge}
h^{1-p}\det(\nabla^2 h+h\,{\rm Id})= g
\end{equation}
for a given non-negative measurable function $g$ on $\Sn$; or more generally, 
given a finite non-trivial measure $\mu$ on $\Sn$, asks for a $K\in \KKn_o$ satisfying
\begin{equation}
\label{LpMinkowskiAlexandrov}
S_{p,K}= \mu.
\end{equation}
For results about the $L_p$-Minkowski problem, 
see for example Chou,  Wang \cite{ChW06}, Chen, Li, Zhu \cite{CLZ17,CLZ19}, 
Bianchi, B\"or\"oczky, Colesanti, Yang \cite{BBCY19} and
Guang, Li, Wang \cite{LGWa}.

The measure $S_{0,K}$ (whose normalized version is the 
Cone Volume measure $V_K=\frac1n\,S_{0,K}$)
was introduced by Firey \cite{Fir74}, and has been a widely used tool since the paper
Gromov, Milman \cite{GromovMilman}, see for example Barthe, Gu\'{e}don, Mendelson, Naor \cite{BG}, Naor \cite{Nar07},  Paouris, Werner \cite{PaW12}.
The still open logarithmic Minkowski problem (\eqref{LpMonge} or \eqref{LpMinkowskiAlexandrov} with $p=0$)  was posed by Firey \cite{Fir74} in 1974, who showed that
if $g$ is a positive constant function, then \eqref{LpMonge} has a unique even solution coming from the suitable centered ball. 
For a positive constant function $g$, the general uniqueness result without the evenness condition is due to  Andrews \cite{And99} if $n=2,3$, and Brendle, Choi, Daskalopoulos \cite{BCD17} if $n\geq 4$ ({\it cf.} 
Andrews, Guan, Ni \cite{AGN16}), see also the arguments
in Saroglou \cite{Sar22} and Ivaki, Milman \cite{IvM23}. 
Concerning general $p$, Hug, Lutwak, Yang, Zhang \cite{HLYZ05} proved that \eqref{LpMinkowskiAlexandrov} has a unique solution if $p>1$, $p\neq n$ and $\mu$ is not concerntrated on any  closed hemisphere 
(see Lutwak \cite{Lut93a} for the case when the  $g$ in \eqref{LpMonge} is a positive contant),
and Brendle, Choi, Daskalopoulos \cite{BCD17} actually verified that if  $-n<p<1$ and $g$ is a positive constant function,
then corresponding ball is the unique solution of \eqref{LpMonge}, see also Saroglou \cite{Sar22}.
It is known that uniqueness of the solution may not hold if $g$ is not a constant function (see 
Chen, Li, Zhu \cite{CLZ17,CLZ19}
for $p\in[0,1)$). Still, it is a fundamental question whether \eqref{LpMonge} has a unique solution if $-n<p<1$
and $g$ is close to be a constant function.

\begin{theorem}[Chen, Feng, Liu \cite{CFL22}]
\label{CFL22}
For $0<\alpha<1$ and $n=3$, there exists a constant $\eta>0$ depending only on $\alpha$  such that if $g\in C^{\alpha}(\mathbb{S}^2)$ satisfies $\|g-1\|_{C^{\alpha}(\mathbb{S}^2)}<\eta$, then the equation
$$
dS_{0,K}=gd\HH^2
$$
has a unique solution in ${\cal K}^2_o$, which is actually a positive $C^{2,\alpha}$ solution of \eqref{LpMonge} with $p=0$.
\end{theorem}

Our main result extends the theorem above for any dimension $n$ and for any $p\in[0,1)$.

\begin{theorem}\label{thm1-bounded-density}
For $n\geq 2$, $\alpha\in(0,1)$ and $p\in[0,1)$, there exists a constant $\eta>0$, that depends only on $n$, $\alpha$, $p$, such that if $g\in C^{\alpha}(\Sn)$ satisfies $\|g-1\|_{C^{\alpha}(\Sn)}<\eta$, then the equation
$$
dS_{p,K}=gd\Ha
$$
has a unique solution in ${\cal K}^n_o$, which is actually a positive $C^{2,\alpha}$ solution of \eqref{LpMonge}.
\end{theorem}

Theorem \ref{thm1-bounded-density} was previously established in \cite{CHLL20} in the symmetric case. The general case, however, appears to be much more challenging.  Ivaki \cite{Iva22} considered the stability of the solution of $L_p$ Monge-Amp\`ere equation \eqref{LpMonge} around 
the constant $1$ function from another point of view. He proved under various conditions depending on $p$ that if $p>-n$
and $g$ is close to the constant $1$ function, then any solution of \eqref{LpMonge} is close to the constant $1$ function.
 In addition, the recent preprint Hu, Ivaki \cite{HuIb} proves a stronger stability result when $p=0$, which in turn yields a second proof of our  Theorem~\ref{thm1-bounded-density}.

Chen, Feng, Liu \cite{CFL22} posed the following related problem:

\begin{problem}[Chen, Feng, Liu \cite{CFL22}]\label{problem 1}
For $n\geq 2$, $p\in[0,1)$ and $\lambda>1$, does there exist a constant $C=C(\lambda,n,p)>1$ such that if $L\in {\cal K}^n_{(o)}$ has absolutely continuous surface area measure and satisfy 
\begin{equation}
\label{eq-main}
\lambda^{-1}\leq h_L^{1-p} f_L\leq \lambda    \qquad \textnormal{on }\Sn,
\end{equation}
then $\sup_{x\in\Sn}h_L(x)\leq C$ and $V(L)\geq C^{-1}$.
\end{problem}

In Section~2 of the paper Jian, Lu, Wang \cite{JLW15}, the authors construct a series of convex bodies showing that 
no suitable constant $C=C(\lambda,n,p)$ exists in Problem~1 if $-n<p<-1$.
We note that Problem~\ref{problem 1} is posed in \cite{CFL22} for $p=0$ (when
$V(L)\geq C^{-1}$ trivially holds). The paper \cite{CFL22} proved that the answer to Problem~\ref{problem 1} is positive for $p=0$ and $n=3$, which statement is one of the main ingredients of the proof of Theorem~\ref{CFL22} in \cite{CFL22}.
Estimates analogues to the one in Problem~\ref{problem 1} are crucial in the study of Monge-Amp\`ere equations, a very similar estimate is obtained for example by 
Huang, Lu \cite{HuL13}, who considered the case when the $C^1$ norm of $h_L^{1-p} f_L$ is bounded.

If $p\in(-1,0)$ and $n=2$, then Du \cite{Du21} proves an analogue of Proplem~\ref{problem 1} in terms of the $C^\alpha$ norm. 
Using the argument completing the proof of Theorem~\ref{thm1-bounded-density} at the end of Section~\ref{secMain},
Du's result in \cite{Du21} leads to the analogue of Theorem~\ref{thm1-bounded-density} for $n=2$,  $\alpha\in(0,1)$ and $p\in(-1,0)$.

In this paper, we clarify Problem~\ref{problem 1} if $p=0$ and $n=4$, and $0<p<1$ and $n=3$.

\begin{theorem}\label{thm2-bounded-density}
The answer to Problem \ref{problem 1} is affirmative if $n\leq 3$ and $p\in[0,1)$ or $n=4$ and $p=0$.    
\end{theorem}

As mentioned, Theorem \ref{CFL22} above relies on the solution to Problem \ref{problem 1} for $n=3$, $p=0$. Although we are unable to clarify Problem \ref{problem 1} in full generality, we manage to prove (see Proposition \ref{prop1-bounded-density} below) that the answer is indeed affirmative if $\lambda$ is chosen to be sufficiently close to 1. This weaker statement turns out to be sufficient for the proof of our Theorem \ref{thm1-bounded-density}. Both Proposition \ref{prop1-bounded-density} and Theorem \ref{thm2-bounded-density} are based on the same general idea: If they are not true, then there exists a convex body $K\in{\cal K}_o^n$, such that ${\rm supp}\,S_{p,K}$ is contained in a proper subspace $L$ of $\R^n$ and the restriction $S_{p,K}\measurerestr{(\Sn\cap L)}$ of $S_{p,K}$ to the class of Borel sets in $\Sn\cap L$ is a Haar measure (resp. is absolutely continuous with respect to the spherical Lebesgue measure in $L$ and its density $\varphi$ is bounded away from 0). In the first case, such body does not exist, due to  \cite[Theorem 1.3]{Sar}. In the second case, we have to deal with the problem of finding conditions for a convex body $K\in{\cal K}_o^n$ to satisfy ${\rm supp}\,S_{p,K}\subset L$, for a proper subspace $L$, and $\varphi$ (as defined previously) to exist and satisfy ${\rm ess\,inf}\,\varphi=0$. We prove that the latter is true if $n,\ p$ are as in Theorem 
\ref{thm2-bounded-density} and that this result is actually optimal.
For a measurable function $\varphi:\R^n\to [0,\infty)$,  we set
$$
{\rm ess\,inf}\,\varphi:=\max_{\omega\subset\Sn,\;\Ha(\omega)=0}\inf \varphi\left|_{\Sn\backslash\omega}\right. .
$$

\begin{theorem}
\label{inf0orpos}
Let $p<1$, let $L\subset \R^n$ be an $(m+1)$-dimensional linear subspace, $0\leq m\leq n-2$, and let $K\in \KKn_o$ satisfy that
${\rm supp}\,S_{p,K}\subset L$, and $d\left(S_{p,K}\measurerestr{(\Sn\cap L)}\right)= \varphi\,d\HH^m$
for a measurable $\varphi:\Sn\cap L\to [0,\infty)$.
\begin{description}
\item{(i)} If $n-2m-p\geq 0$, then ${\rm ess\,inf}\,\varphi=0$;
\item{(ii)} If $n-2m-p<0$, then one may choose $K$ in a way such that $\varphi$ is a positive continuous function.
\end{description}
\end{theorem}

For negative $p$, we do not have uniqueness results in the spirit of Theorem \ref{thm1-bounded-density}. Instead, employing a simple variational argument, we establish a strong non-uniqueness result which is valid for both the symmetric and the non-symmetric case.
Set ${\cal C}_+^\infty$ to be the set of convex bodies whose boundary is $C^\infty$ with strictly positive curvature, and hence their support functions are $C^\infty$ on $\R^n\backslash\{o\}$. The following theorem extends results from \cite{JLW15} and \cite{Mila}.

\begin{theorem}\label{thm-non-uniqueness}
Let ${\cal K}={\cal K}_o^n$ or ${\cal K}={\cal K}_s^n$, and let
${\cal K}_+:={\cal K}\cap {\cal C}_+^\infty$ and $p\in(-n,0)$. Then the set 
$${\cal A}_p:=\{K\in{\cal K}_+:o\in{\rm int}\,K\textnormal{ and }\forall q\in(-n,p],\ \exists L_q\in{\cal K}, \textnormal{ such that }L_q\neq K\textnormal{ and }S_{q,K}=S_{q,L_q}\}$$ is dense in ${\cal K}$, in the Hausdorff metric.
\end{theorem}
\begin{remark}\label{rem-after-thm-non-uniqueness}
Standard regularity theory shows that if ${\cal K}={\cal K}_s^n$, then each $L_q$ that appears in the definition of ${\cal A}_p$ is also a member of ${\cal K}_+$. Unfortunately, in the non-symmetric case, since the origin can lie in the boundary of $L_q$ (see \cite{JLW15}), we can only conclude that $L_q$ is a generalized solution of the equation $S_{q,K}=S_{q,L_q}$. Moreover, the theorem is still valid, if one replaces the condition ``$C^\infty$ boundary" by ``real analytic boundary", in the definition of ${\cal K}_+$. This raises the following question: If $q\in(-n,0)$, $S_{q,L}=fdx$ and $f>0$ is real analytic, can $o\in\partial L$?
\end{remark}

We note that question of uniqueness of the solution has been discussed in various versions of the Minkowski problem.
The Gaussian surface area measure of a $K\in\mathcal{K}^n$ is defined by
Huang, Xi and Zhao \cite{HXZ21}, whose results are extended  by Feng, Liu, Xu \cite{FLX23}, Liu \cite{Liu22} and Feng, Hu, Xu \cite{FHX23}. The fact that only balls are the solutions when the Gaussian curvature is constant 
is proved by Chen, Hu, Liu, Zhao \cite{CHLZ} for convex domains in $\R^2$, and in the even case for $n\geq 3$
by \cite{CHLZ} and Ivaki, Milman \cite{IvM23}. The intensively investigated $L_p$-Minkowski conjecture
stated by B\"or\"oczky, Lutwak, Yang,  Zhang \cite{BLYZ12}
claims the uniqueness of the even solution of \eqref{LpMonge} for even positive $g$,
see for example
B\"or\"oczky, Kalantzopoulos \cite{BoK22},
Chen, Huang, Li,  Liu \cite{CHLL20},
Colesanti,  Livshyts, Marsiglietti \cite{CLM17},
Colesanti,  Livshyts \cite{CoL20},
Ivaki, E. Milman \cite{IvM23b}, 
Kolesnikov \cite{Kol20},
Kolesnikov, Livshyts \cite{KoL},
Kolesnikov, Milman \cite{KoM22},
Livshyts, Marsiglietti, Nayar, Zvavitch \cite{LMNZ20},
Milman \cite{Mila,Milb}, 
Saroglou \cite{Sar15,Sar16},
Stancu \cite{Stancu,Stancu1,Sta22},
van Handel \cite{vHa} 
for partial results, and B\"or\"oczky \cite{Bor23} for a survey of the subject.

For the $L_p$ $q$th dual Minkowski problem
due to Lutwak, Yang, Zhang \cite{LYZ18} (see Huang, Lutwak, Yang, Zhang \cite{HLYZ16}
for the original $q$th dual Minkowski problem), Chen, Li \cite{ChL21} and
Lu, Pu \cite{LuP21} solve essentially the case $p>0$, and Huang, Zhao \cite{HuZ18} and Guang, Li, Wang \cite{LGW23a} discuss the case $p<0$ (see Gardner, Hug, Weil, Xing, Ye \cite{GHWXY19,GHXY20} for Orlicz versions of some of these results). Uniqueness of the solution of the $q$th $L_p$ dual Minkowski problem is thoroughly investigated by Li, Liu, Lu \cite{LLL22}. The case when $n=2$ and $f$ is a constant function has been completely clarified by
Li,  Wan \cite{LiW}. If the $L_p$ $q$th dual curvature is constant, then Ivaki, Milman \cite{IvM23} shows uniqueness of the even solution when $p>-n$ and $q\leq n$.

Concerning the structure of the paper, we start with the proof of
Theorem~\ref{inf0orpos} because some of the ingredients are used in the proof of Theorem~\ref{thm1-bounded-density},
as well. Theorem~\ref{thm1-bounded-density} is proved in
Section~\ref{secMain}, and  Theorem~\ref{thm2-bounded-density} is verified in 
Section~\ref{secBoundedDensity}. Finally, the severe non-uniqueness result for  $-n<p<0$; namely, Theorem~\ref{thm-non-uniqueness} is proved in Section~\ref{secNonUnique}.

\section{When the support of the $L_p$ surface area measure is lower dimensional and the proof of Theorem~\ref{inf0orpos} }

If $L$ is an $(m+1)$-dimensional linear subspace of $\R^n$, $0\leq m\leq n-2$, and $X\subset \R^n$, then $X|L$ denotes the orthogonal projection of $X$ into $L$.
If $M\subset L$ is a full dimensional compact convex set (the affine hull of $M$ is $L$), then we write $\partial M$ to denote the relative boundary  of $M$ with respect to $L$, and $\partial' M$ to denote the set of points $x\in\partial M$ such that
there exists a unique exterior normal $\nu_M(x)\in L$ to $M$ at $x$. In particular, $\HH^m(\partial M\backslash \partial' M)=0$.
For $\Omega\in{\rm GL}\,(L)$ and Borel measure $\mu$ on $\Sn\cap L$, the push forward measure 
$\Omega_*\mu$ on $\Sn\cap L$ is defined in a way such that for a Borel set $\omega\subset \Sn\cap L$,
$$
\Omega_*\mu(\omega)=\mu\left\{\frac{\Omega^{-1} u}{\|\Omega^{-1} u\|}:u\in\omega\right\},
$$
which satisfies that if $\psi:\Sn\cap L\to [0,\infty)$ is measurable, then
$$
\int_{\Sn\cap L}\psi\,d\Omega_*\mu=
\int_{\Sn\cap L}\psi\left(\frac{\Omega u}{\|\Omega u\|}\right)\,d\mu(u).
$$

\begin{lemma}
\label{structure}
Let $L\subset \R^n$ be an $(m+1)$-dimensional linear subspace, $0\leq m\leq n-2$, and let 
$K\in \KKn_o$ such that
${\rm supp}\,S_{p,K}\subset L$, and let $M=K|L$.
\begin{description}
\item{(i)} $o\in\partial K\cap \partial M$,  and $K=C\cap (M+L^\bot)$ for the convex cone \\
$C=\{x\in\R^n:\,\langle x,v\rangle\leq 0\mbox{ \ for \ }
v\in({\rm supp}\,S_K)\backslash L\}$;
\item{(ii)} $L^\bot\cap C=\{o\}$;
\item{(iii)} For any Borel set $\omega\subset \Sn\cap L$, we have
$$
S_{p,K}(\omega)=\int_{\nu_M^{-1}(\omega)}\langle x,\nu_M(x)\rangle^{1-p} \HH^{n-m-1}\Big(C\cap(x+L^\bot)\Big)\,d\HH^m(x);
$$
\item{(iv)} If $\Phi=\widetilde{\Phi}\oplus\Psi\in{\rm GL}(\R^n)$ for $\widetilde{\Phi}\in{\rm GL}(L)$ and
$\Psi\in{\rm GL}(L^\bot)$, then  ${\rm supp}\, S_{p,\Phi K}\subset L$, 
and there exist $\gamma_2>\gamma_1>0$ depending on $p,\Phi$ such that 
if $\omega\subset \Sn\cap L$ is Borel
and $\widetilde{\omega}=\left\{\frac{\widetilde{\Phi}^{-t}u}{\|\widetilde{\Phi}^{-t}u\|}:\, u\in\omega \right\}$, then
$\gamma_1S_{p,K}(\omega)\leq S_{p,\Phi K}(\widetilde{\omega})\leq \gamma_2 S_{p,K}(\omega)$.
\end{description}
\end{lemma}
\proof Since $dS_{p,K}=h^{1-p}dS_K$, it follows that
\begin{equation}
\label{suppSKp}
 {\rm supp}\,S_{p,K}={\rm cl}\,\{v\in {\rm supp}\,S_K:\;h_K(v)> 0\}.
\end{equation}
According to Minkowski's theorem ({\it cf.} Schneider \cite{Sch14}),
${\rm supp}\,S_K$ is not contained in any closed hemisphere, and hence not contained in $L$. 
As ${\rm supp}\,S_{p,K}\subset L$, \eqref{suppSKp} yields that $h_K(v)=0$ for any 
$v\in ({\rm supp}\,S_K)\backslash L$. 

For $v\in \Sn$, let $H_v^+=\{x\in\R^n:\,\langle x,v\rangle\leq h_K(v)\}$; therefore, 
({\it cf.} Schneider \cite{Sch14}) 
\begin{equation}
\label{suppSK}
K=\cap\{H_v^+:\,v\in {\rm supp}\,S_K\}.
\end{equation}
Since $C=\cap\{H_v^+:\,v\in ({\rm supp}\,S_K)\backslash L\}$, and 
$M+L^\bot=\cap\{H_v^+:\,v\in ({\rm supp}\,S_K)\cap L\}$, we conclude (i). In addition, (ii) follows from the fact that for any $w\in \Sn\cap L^\bot$, there exists a $v\in ({\rm supp}\,S_K)\backslash L$ with $\langle v,w\rangle>0$.

Turning to (iii), for any $x\in\partial K$, we write $\tilde{\nu}_K(x)$ to denote the spherical image of $x$; namely, the set of all exterior unit normal vectors at $x$ to $K$. In particular, $\tilde{\nu}_K(x)=\{\nu_K(x)\}$ if $x\in\partial' K$. Similarly, for any $x\in\partial M$, we write $\tilde{\nu}_M(x)$ the set of all exterior
normals $v\in \Sn\cap L$ at $x$ to $M$. It folows that if $x\in\partial M$ and $y$ lies in the relative interior of 
$(x+L^\bot)\cap \partial K$, then $\tilde{\nu}_M(x)=\tilde{\nu}_K(y)$. Therefore, $x\in \partial' M$ for $x\in \partial M$ if and only if there exists a $y\in (x+L^\bot)\cap \partial' K$, and if $x\in \partial' M$, then $\nu_M(x)=\nu_K(y)$ for $\HH^{n-m-1}$ almost all $y\in  (x+L^\bot)\cap \partial' K$. In turn, 
\eqref{SKintegral} and the Fubini theorem yield (iii). 

For (iv), let $\Phi=\widetilde{\Phi}\oplus\Psi\in{\rm GL}(\R^n)$ for $\widetilde{\Phi}\in{\rm GL}(L)$ and
$\Psi\in{\rm GL}(L^\bot)$. After rescaling $\Phi$, we may assume that $\det \widetilde{\Phi}=1$.
We note that if $x\in\partial' M$, then $\widetilde{\Phi}x\in\partial' (\widetilde{\Phi} M)$, and 
$\nu_{\widetilde{\Phi}M}(\widetilde{\Phi}x)=\frac{\widetilde{\Phi}^{-t}\nu_{M}(x)}{\|\widetilde{\Phi}^{-t}\nu_{M}(x)\|}$.
It follows from (iii) that if $\omega \subset \Sn$ is Borel, then
$$
S_{p,\Phi K}(\omega)=\int_{\nu_{\widetilde{\Phi}M}^{-1}(\omega\cap L)}\langle z,\nu_{\widetilde{\Phi}M}(z)\rangle^{1-p} 
\HH^{n-m-1}\Big(\Phi C\cap(z+L^\bot)\Big)\,d\HH^m(z);
$$
therefore, ${\rm supp}\, S_{p,\Phi K}\subset L$.

The invariance of the normalized cone volume measure $S_{0,M}$ under the special linear group (see Firey \cite{Fir74}.
B\"or\"oczky, Lutwak, Yang, Zhang \cite{BLYZ12,BLYZ13}, B\"or\"oczky \cite{Bor23}) yields that if $\Phi_0\in{\rm GL}(\R^n)$, then 
$$
S_{0,\Phi_0M}=|\det \Phi_0|\left(\Phi_0^{-t}\right)_*S_{0,M}.
$$
Thus, if $\omega\subset \Sn\cap L$ is Borel, then $\widetilde{\omega}=\left\{\frac{\widetilde{\Phi}^{-t}u}{\|\widetilde{\Phi}^{-t}u\|}:\, u\in\omega \right\}=\left\{\frac{\Phi^{-t}u}{\|\Phi^{-t}u\|}:\, u\in\omega \right\}$ and it holds
\begin{eqnarray*}
S_{0,\Phi M}(\widetilde{\omega})
&=&|\det \Phi|\cdot S_{0,M}(\{\Phi^tu/\|\Phi^tu\|:u\in\widetilde{\omega}\})\\
&=&|\det \Psi|\cdot S_{0,M}(\{\widetilde{\Phi}^tu/\|\widetilde{\Phi}^tu\|:u\in\widetilde{\omega}\})=|\det\Psi|\cdot S_{0,M}(\omega).
\end{eqnarray*}
Let $\Xi=\{u\in \Sn\cap L:h_M(u)>0\}$, and hence
$$
\widetilde{\Xi}=\{u\in \Sn\cap L:h_{\widetilde{\Phi}M}(u)>0\}=
\left\{\frac{\widetilde{\Phi}^{-t}u}{\|\widetilde{\Phi}^{-t}u\|}:\, u\in\Xi \right\}.
$$
We note that if $u\in \Sn\cap L$, then 
$h_{\widetilde{\Phi}M}\left(\frac{\widetilde{\Phi}^{-t}u}{\|\widetilde{\Phi}^{-t}u\|}\right)=
\frac{1}{\|\widetilde{\Phi}^{-t}u\|}\cdot h_M(u)$.
It follows from (iii), that if $\omega\subset \Sn\cap L$ is Borel
and $\widetilde{\omega}=\left\{\frac{\widetilde{\Phi}^{-t}u}{\|\widetilde{\Phi}^{-t}u\|}:\, u\in\omega \right\}$, then
$$
S_{p,\Phi K}(\widetilde{\omega})=
\int_{\widetilde{\omega}\cap \widetilde{\Xi}}h_{\widetilde{\Phi} M}^{-p}\,dS_{0,\Phi K}=
|\det\Psi|\cdot \int_{\omega\cap \Xi}\|\widetilde{\Phi}^{-t}u\|^{p}h_{M}(u)^{-p}\,dS_{0,K}(u).
$$
As $u\mapsto\|\widetilde{\Phi}^{-t}u\|^{p}$ is positive and continuous on $\Sn\cap L$, there exist
$\tilde{\gamma}_2>\tilde{\gamma}_1>0$ such that
$\tilde{\gamma}_1\leq \|\widetilde{\Phi}^{-t}u\|^{p}\leq \tilde{\gamma}_2$ for $u\in \Sn\cap L$.
Since $S_{p,K}(\omega)=\int_{\omega\cap \Xi}h_{M}^{-p}\,dS_{0,K}$, we conclude (iv)
with $\gamma_1=|\det\Psi|\cdot \tilde{\gamma}_1$ and $\gamma_2=|\det\Psi|\cdot \tilde{\gamma}_2$.
\proofbox

Let $L\subset \R^n$ be an $(m+1)$-dimensional linear subspace, $0\leq m\leq n-2$, and let 
$K\in \KKn_o$.  The symmetral $K'\in \KKn_o$ of $K$ through $L$ is defined in a way such that 
 $(x+L^\bot)\cap K$ is a ball centered at $x$ 
satisfying $\HH^{n-m-1}(K'\cap(x+L^\bot))=\HH^{n-m-1}(K\cap(x+L^\bot))$
for any $x\in ({\rm int}\,K)|L$.
The fact the $K'$ is convex follows from the Brunn-Minkowski inequality. In turn, we say that an $n$-dimensional closed convex set $X\subset \R^n$ 
is symmetric through $L$ if $(x+L^\bot)\cap X$ is either  $x+L^\bot$ or an $(n-m-1)$-dimensional ball centered at $x$, for any $x\in ({\rm int}\,X)|L$ .

The following consquence of Lemma~\ref{structure} (iii) says that we need to deal with only convex  bodies with reasonable symmetry.

\begin{lemma}
\label{symmetrization}
Let $L\subset \R^n$ be an $(m+1)$-dimensional linear subspace, $0\leq m\leq n-2$, and let 
$K\in \KKn_o$. If
${\rm supp}\,S_{p,K}\subset L$ and $K'\in \KKn_o$ is the symmetral of $K$ through $L$, then
$S_{p,K}=S_{p,K'}$.
\end{lemma}

Let $C\subset \R^n$ be an $n$-dimensional closed convex cone not containing a halfspace of $\R^n$, and let $\Pi$ be the largest linear subspace contained in $C$, and hence $0\leq {\rm dim}\,\Pi\leq n-2$. Then $C_0=C\cap \Pi^\bot$ is a convex cone in $\Pi^\bot$ not containing a line satisfying $C=C_0+\Pi$. For the dual cone
$C^*=\{y\in \R^n:\langle x,y\rangle \leq 0\mbox{ \ for any $x\in C$}\}$ consisting of the exterior normals to $C$ at $o$, we have
$C^*\subset \Pi^\bot$ with ${\rm dim}\,C^*={\rm dim}\, \Pi^\bot$. It follows that 
\begin{equation}
\label{coneSectionCompact}
\mbox{if $u\in \Sn\cap{\rm relint}\,C^*$
and $r>0$, then $(-ru+u^\bot)\cap C_0$ \ is compact.}
\end{equation}
In turn, we deduce from \eqref{coneSectionCompact} the following statement.

\begin{lemma}
\label{coneSection}
Let $L\subset \R^n$ be an $(m+1)$-dimensional linear subspace, $0\leq m\leq n-2$, and let 
$C$ be an $n$-dimensional closed convex cone symmetric around $L$ such that
$L^\bot\cap C=\{o\}$. 
\begin{description}
\item{(i)} The maximal linear subspace $\Pi\subset C$ satisfies $\Pi\subset L$;
\item{(ii)} There exists $u\in \Sn\cap L\cap{\rm relint}\,C^*$, that is equivalent saying that
$u\in \Sn\cap L\cap C^*$ and $u^\bot\cap C=\Pi$;
\item{(iii)} If  $u\in \Sn\cap L\cap{\rm relint}\,C^*$, then there exists $\gamma>0$ depending on $u$, $L$, $C$ such that $\HH^{n-m-1}\Big((x+L^\bot)\cap C\Big)\leq \gamma|\langle x,u\rangle|^{n-m-1}$
for any $x\in C$.
\end{description}
\end{lemma}

We say that $\varphi$ is $C^\alpha$ if there exists $\alpha\in(0,1)$ and $C>0$ such that
$|\varphi(x)-\varphi(y)|\leq C\|x-y\|^\alpha$, for all $x,y\in\Sn$.

For $P\in \KK_{(o)}^m$ in $\R^m$, we define its polar dual $P^*\in \KK_{o}^m$ to be 
$P^*=\{y\in \R^m:\langle y,x\rangle\leq 1\mbox{ \ for \ }x\in P\}$. It is well known
 that there exists $\tau_m>0$ depending on $m$ 
(see Schneider \cite{Sch14}, and the arxiv version of Kuperberg \cite{Kup08} for the best factor 
$\tau_m=4^{-m}\HH^m(B^m)^2$ known) such that
\begin{equation}
\label{Mahler}
\HH^m(P^*)\geq \frac{\tau_m}{\HH^m(P)} \mbox{ \  for any $P\in \KK_{(o)}^m$.}
\end{equation}
For any convex $X\subset\R^n\backslash\{o\}$, its positive hull is the convex cone
${\rm pos}X=\{\lambda x:x\in X\mbox{ and }\lambda\geq 0\}$.

{\it Proof of  Theorem~\ref{inf0orpos} (i).} 
According to Lemma~\ref{symmetrization}, we may assume that $K$ is symmetric through $L$.
Let $M=K|L$. It follows from Lemma~\ref{structure} that
$o\in\partial K$, and $K=C\cap (M+L^\bot)$ and $L^\bot\cap C=\{o\}$ for the convex cone 
$C=\{x\in\R^n:\,\langle x,v\rangle\leq 0\mbox{ \ for \ }
v\in({\rm supp}\,S_K)\backslash L\}$. We observe that $C$ is also symmetric through $L$.

Let $\Pi\subset C$ be the maximal linear subspace, and hence $\Pi\subset L$.
According to Lemma~\ref{coneSection}, there exists $u\in \Sn\cap L\cap{\rm relint}\,C^*$, and
the $m$-dimensional linear subspace $L_0=u^\bot\cap L$ satisfies $L_0\cap C=\Pi$.
We fix an $x_0\in{\rm relint}\,M$, and let $\Phi\in {\rm GL}(\R^n)$ be such that
$\Phi$ acts as the identity on $L^\bot$ and on $L_0$, and $\Phi(x_0)=-u$.
In particular, $-u\in {\rm relint}\,\Phi M$, $L^\bot\cap \Phi C=\{o\}$  and $L_0\cap \Phi C=\Pi$.
It follows from Lemma~\ref{structure} (iv) that we may replace $K$ by $\Phi K$, therefore, we may assume that $-u\in {\rm relint}\,M$ and $L_0\cap C=\Pi$ for $L_0=u^\bot\cap L$.
Let $\varrho\in(0,\frac12)$ such that $(-u+\varrho\,B^n)\cap L\subset M$.

For $r\in[0,\frac12)$, let $M_r= M\cap (-ru+L_0)$, and for $r\in(0,\frac12)$, let
$$
Q_r=\{y\in L:\langle y,x\rangle\leq 0\mbox{ \ for \ }x\in M_r\}.
$$
It follows from $u\in {\rm relint}\,C^*$ that $u\in {\rm relint}Q_r$.

We claim that if $n-2m-p\geq 0$, then $\omega_r=\Sn\cap{\rm relint}\,Q_r$ satisfies
\begin{equation}
\label{phiusmallclaim}
\lim_{r\to 0^+}\frac{S_{p,K}(\omega_r)}{\HH^m(\omega_r)}=0.
\end{equation}
Readily, \eqref{phiusmallclaim} yields that ${\rm ess\,inf}\,\varphi=0$.

Let $r\in(0,\frac12)$. We observe that the convexity of $M$ and $o\in M$ yield that  
\begin{equation}
\label{farpoints}
\{x\in M:\langle x,-u\rangle\geq r\}\subset {\rm pos}\,M_r.
\end{equation}
Thus $(-u+\varrho\,B^n)\cap L\subset {\rm pos}\,M_r$, and hence
\begin{equation}
\label{Qrnarrow}
\langle y,u\rangle\geq \varrho\cdot\|y\| \mbox{ \ for $y\in Q_r$.}
\end{equation}

To prove \eqref{phiusmallclaim}, first  we provide a lower bound for $\HH^m(\omega_r)$ for $r\in(0,\frac12)$. Since $-u\in{\rm relint}\,M$, it follows that
$$
o\in {\rm relint}(M_r|L_0)={\rm relint}(M_r+ru).
$$
For $\bar{Q}_r=Q_r\cap (u+L_0)$, we observe that 
$$
\bar{Q}_r-u=\{y\in L_0:\langle y,x\rangle\leq r\;\forall x\in M_r+ru\}=r(M_r+ru)^*,
$$ 
and hence \eqref{Mahler} implies that
$$
\HH^m(\bar{Q}_r)\geq \frac{\tau_mr^m}{\HH^m(M_r)}.
$$
Now $\omega_r=\Sn\cap{\rm relint}Q_r$ is the radial projection of ${\rm relint}\bar{Q}_r$ into $\Sn\cap L$; therefore,
\eqref{Qrnarrow} implies that
\begin{equation}
\label{omegarlow}
\HH^m(\omega_r)\geq \tau_m\varrho^{m+1}\cdot \frac{r^m}{\HH^m(M_r)}.
\end{equation}

Next we turn to our harder task in order to verify \eqref{phiusmallclaim}; namely, to provide a suitable upper bound for
$S_{K,p}(\omega_r)$. Let $r\in(0,\frac12)$. For an $x\in \partial' M$ with $\nu_M(x)\in {\rm relint}Q_r$, the definition of $Q_r$ yields that $\langle \nu_M(x),z\rangle<0$ for any $z\in {\rm pos}\,M_r$ with $z\neq o$. Since
$\langle \nu_M(x),x\rangle=h_M(\nu_M(x))\geq 0$, we deduce from 
\eqref{farpoints} that $\langle x,-u\rangle< r$. On the other hand, the convexity of $M$ and $-u\in M$ yield that
$$
\{x\in M:\langle x,-u\rangle\leq r\}\subset -u+{\rm pos}\,(M_r+u)
$$
where $-u+{\rm pos}\,(M_r+u)$ consists of the half lines emanating from $-u$ and passing through $M_r$. Therefore,
\begin{equation}
\label{omegarMr}
\nu_M^{-1}(\omega_r)\subset 
\left(-u+{\rm pos}\,(M_r+u)\right)\cap \{x\in L:0\leq \langle x,-u\rangle\leq r\}=
{\rm conv}\left\{M_r,\mbox{$\frac1{1-r}\,$}(M_r|L_0)\right\}.
\end{equation}

Let us summarize what we know so far about $\nu_M^{-1}(\omega_r)$ for $r\in(0,\frac12)$.
We deduce from \eqref{omegarMr} and $M$ being closed that for $\varepsilon>0$, there exists
$r_\varepsilon\in(0,\frac12)$ such that if $0<r<r_\varepsilon$, then
\begin{equation}
\label{omegarM0}
\left.\nu_M^{-1}(\omega_r)\right|L_0\subset 
\mbox{$\frac1{1-r}\,$}(M_r|L_0)\subset M_0+\varepsilon(B^n\cap L_0).
\end{equation}
For $x\in \nu_M^{-1}(\omega_r)$, it follows from \eqref{Qrnarrow} that
\begin{equation}
\label{omegarnuM}
\langle \nu_M(x),u\rangle\geq \varrho,
\end{equation}
and Lemma~\ref{coneSection} (iii) and \eqref{omegarMr} imply that
\begin{equation}
\label{omegarC}
\HH^{n-m-1}\Big((x+L^\bot)\cap C\Big)\leq \gamma\,r^{n-m-1}
\end{equation}
where $\gamma>0$ depends on $C$ and $u$. For $x\in \nu_M^{-1}(\omega_r)$ and $x|L_0=z$, we consider the function $g(z)=\langle \nu_M(x),u\rangle^{-1}\leq \frac1{\varrho}$ that is continuous on 
$\nu_M^{-1}(\omega_r)|L_0$ and (in view of \eqref{omegarnuM}) satisfies
\begin{equation}
\label{omegarg}
\HH^{m}\Big(\nu_M^{-1}(\omega_r)\Big)=\int_{\nu_M^{-1}(\omega_r)|L_0}g\,d\HH^{m}
\leq \mbox{$\frac1{\varrho}\,$}\HH^{m}\Big(\nu_M^{-1}(\omega_r)|L_0\Big);
\end{equation}
therefore, \eqref{omegarMr} and $\frac1{1-r}\,(M_r|L_0)\subset 2(M_r+ru)$ yield that 
\begin{equation}
\label{omegarbounded}
\HH^{m}\Big(\nu_M^{-1}(\omega_r)\Big)\leq \mbox{$\frac{2^m}{\varrho}\,$}\HH^{m}(M_r).
\end{equation}
As our last auxialiary estimate, we still need an upper bound on $\langle \nu_M(x),x\rangle$ for
 $x\in \nu_M^{-1}(\omega_r)$ and $r\in(0,\frac12)$. It follows from
\eqref{omegarMr} and $\frac1{1-r}\,(M_r|L_0)\subset 2(M_r+ru)$ that
$\nu_M^{-1}(\omega_r)\subset 2ru+{\rm pos}\,M_r$; therefore, $\nu_M(x)\in Q_r$ yields
\begin{equation}
\label{omegarnumupp}
\langle \nu_M(x),x\rangle\leq \langle \nu_M(x),2ru\rangle\leq 2r.
\end{equation}

The final part of the argument leading to the claim \eqref{phiusmallclaim} is divided into two cases.\\

\noindent{\bf Case 1.} Either $n-2m-p>0$ or $n-2m-p=0$ and ${\rm dim}\,M_0\leq m-1$.\\
It follows from Lemma~\ref{structure} (iii), \eqref{omegarC}, 
\eqref{omegarbounded} and \eqref{omegarnumupp} that if $r\in(0,\frac12)$, then
\begin{eqnarray*}
S_{p,K}(\omega_r)&=&\int_{\nu_M^{-1}(\omega_r)}\langle x,\nu_M(x)\rangle^{1-p} \HH^{n-m-1}\Big(C\cap(x+L^\bot)\Big)\,d\HH^m(x)\\
&\leq & \mbox{$\frac{2^m}{\varrho}\,$}\HH^{m}(M_r)\cdot (2r)^{1-p}\cdot \gamma\,r^{n-m-1}=
\eta_1r^{n-m-p}\cdot\HH^{m}(M_r)
\end{eqnarray*}
for a constant $\eta_1>0$ depending on $M$, $\varrho$ and $u$; therefore,
\begin{equation}
\label{Case1est}
\frac{S_{p,K}(\omega_r)}{\HH^{m}(\omega_r)}\leq \eta_2r^{n-2m-p}\cdot\HH^{m}(M_r)^2
\end{equation}
follows from \eqref{omegarlow}
for a constant $\eta_2>0$ depending on $M$, $\varrho$ and $u$.

If $n-2m-p>0$, then combining \eqref{Case1est} and $\HH^{m}(M_r)\leq \HH^{m}(M|L_0)$ yields 
the claim \eqref{phiusmallclaim}. If $n-2m-p=0$ and ${\rm dim}\,M_0\leq m-1$, then
\eqref{omegarM0} implies that $\lim_{r\to 0^+}\HH^{m}(M_r)=0$, and in turn 
\eqref{Case1est} implies \eqref{phiusmallclaim}.\\

\noindent{\bf Case 2.} $n-2m-p=0$ and ${\rm dim}\,M_0=m$.\\
In this case, if $x\in {\rm relint}M_0$, then $x\in\partial' M$ with $\nu_M(x)=u$, and hence
$\langle x,\nu_M(x)\rangle=0$. It follows from Lemma~\ref{structure} (iii),
\eqref{omegarC} and \eqref{omegarnumupp} that if $r\in(0,\frac12)$, then
\begin{eqnarray*}
S_{p,K}(\omega_r)&=&\int_{\nu_M^{-1}(\omega_r)\backslash ({\rm relint}M_0)}\langle x,\nu_M(x)\rangle^{1-p} \HH^{n-m-1}\Big(C\cap(x+L^\bot)\Big)\,d\HH^m(x)\\
&\leq & \HH^{m}\left(\nu_M^{-1}(\omega_r)\backslash ({\rm relint}M_0)\right)\cdot 
(2r)^{1-p}\cdot \gamma\,r^{n-m-1}=
\eta_3r^{n-m-p}\cdot\HH^{m}\left(\nu_M^{-1}(\omega_r)\backslash ({\rm relint}M_0)\right)
\end{eqnarray*}
for a constant $\eta_3>0$ depending on $M$, $\varrho$ and $u$. We deduce from \eqref{omegarlow},
$n-2m-p=0$ and $\HH^{m}(M_r)\leq \HH^{m}(M|L_0)$ that
\begin{equation}
\label{Case2est}
\frac{S_{p,K}(\omega_r)}{\HH^{m}(\omega_r)}\leq \eta_4
\HH^{m}\left(\nu_M^{-1}(\omega_r)\backslash ({\rm relint}M_0)\right).
\end{equation}
Analogously to \eqref{omegarg}, we have
$$
\HH^{m}\Big(\nu_M^{-1}(\omega_r)\backslash ({\rm relint}M_0)\Big)=
\int_{\Big(\nu_M^{-1}(\omega_r)|L_0\Big)\backslash ({\rm relint}M_0)}g\,d\HH^{m}
\leq \mbox{$\frac1{\varrho}\,$}\HH^{m}\Big((\nu_M^{-1}(\omega_r)|L_0)\backslash ({\rm relint}M_0)\Big)
$$
where \eqref{omegarM0} implies that
$$
\lim_{r\to 0^+}\HH^{m}\Big((\nu_M^{-1}(\omega_r)|L_0)\backslash ({\rm relint}M_0)\Big)=0.
$$
Therefore, \eqref{Case2est} yields 
the claim \eqref{phiusmallclaim}.

In turn, the claim \eqref{phiusmallclaim} implies that ${\rm ess\,inf}\,\varphi=0$, and hence Theorem~\ref{inf0orpos} (i), as well. 
\proofbox

\bigskip

\begin{proof}[Proof of  Theorem~\ref{inf0orpos} (ii)] We consider a linear $(m+1)$-space $L$ of $\R^n$, $0\leq m\leq n-2$, and $p<1$ such that $n-2m-p<0$. Let
$$
q=\frac{2m}{2m+p-n}>2.
$$
We fix a $u\in \Sn\cap L$, set $L_0=u^\bot\cap L$, and consider an $(m+1)$-dimensional convex compact set $M\subset L$ with $C^2$ (relative) boundary such that $o\in \partial M$, $\nu_M(o)=u$, 
the Gaussian curvature $\kappa_M(x)$ at an $x\in \partial M$ satisfies $\kappa_M(x)>0$ for $x\neq o$,
 there exists $r>0$ such that
$$
z-g(z)u\in \partial M\mbox{ \ for $z\in L_0\cap rB^n$ and $g(z)=\|z\|^q$}.
$$
We use the notation $x_v=\nu_M^{-1}(v)$
for $v\in \Sn\cap L$.

To construct the convex body $K$ in $\R^n$, we consider the convex cones 
$C_0=\{x\in L_0^\bot: \langle x,-u\rangle \geq \frac1{\sqrt{2}}\,\|x\|\}$ and $C=C_0+L_0$. In particular, $C$ is an $n$-dimensional cone symmetric through $L$, and if
$x\in C$, then the radius of the $(n-m-1)$-dimensional ball $(x+L^\bot)\cap C$ is 
$\langle x,-u\rangle$. We define
$$
K=C\cap(M+L^\bot),
$$
and set $\eta_k=\HH^k(B^k)$.
Readily, ${\rm supp}\,S_{p,K}\subset L$, and Lemma~\ref{structure} (iii) yields that  
$d\left(S_{p,K}\measurerestr({\Sn\cap L})\right)=\varphi\,d\HH^m$ where if $v\in (\Sn\cap L)\backslash\{u\}$, then
\begin{equation}
\label{varphidef}
\varphi(v)=\langle x_v,v\rangle^{1-p}\cdot \eta_{n-m-1}\langle x_v,-u\rangle^{n-m-1}\cdot \kappa_M(x_v)^{-1}.
\end{equation}
In particular, $\varphi$ is continuous and positive on $(\Sn\cap L)\backslash\{u\}$.
Therefore, setting $v(z)=\nu_M(z-g(z)u)$ for $z\in L_0\cap rB^n$, Theorem~\ref{inf0orpos} (ii) follows from the claim
\begin{equation}
\label{varphiatu}    
\lim_{v\to u\atop v\in \Sn\cap L} \varphi(v)=
\lim_{z\to o\atop z\in L_0\cap rB^n} \varphi(v(z))
=\frac{\eta_{n-m-1}}{q^m(q-1)^p}.
\end{equation}
Here we deduce from \eqref{varphidef}, $x_{v(z)}=z-g(z)u$ and 
$\langle x_{v(z)},-u\rangle=g(z)=\|z\|^q$ that if $z\in L_0\cap rB^n$, then
\begin{equation}
\label{varphivz}
\varphi(v(z))=\langle z-g(z)u,v(z)\rangle^{1-p}\cdot \eta_{n-m-1}\|z\|^{q(n-m-1)}\cdot \kappa_M(z-g(z)u)^{-1}.
\end{equation}

We write $Dg$ and $D^2g$ to denote the gradient and the Hessian of $g$ where  if $z\in L_0\cap rB^n$, then
\begin{equation}
\label{gradientg}
Dg(z)=q\|z\|^{q-2}\cdot z, \mbox{ \ and hence \ }\lim_{z\to o}Dg(z)=o.
\end{equation} 
It follows from \eqref{gradientg} that
\begin{eqnarray*}
v(z)&=&\left(1+\|Dg(z)\|^2\right)^{-\frac12}(Dg(z)+u);\\
\langle z-g(z)u,v(z)\rangle&=&\left(1+\|Dg(z)\|^2\right)^{-\frac12}(q-1)\|z\|^q;\\
\kappa_M(z-g(z)u)^{-1}&=&\frac{\left(1+\|Dg(z)\|^2\right)^{\frac{m+2}2}}{\det D^2g(z)}
=\frac{\left(1+\|Dg(z)\|^2\right)^{\frac{m+2}2}}{q^m(q-1)\|z\|^{m(q-2)}}.
\end{eqnarray*}
We deduce from \eqref{gradientg} and $2m+q(n-2m-p)=0$ that
$$
\lim_{z\to o\atop z\in L_0\cap rB^n} \frac{\varphi(v(z))}{\eta_{n-m-1}}=
\lim_{z\to o\atop z\in L_0\cap rB^n}
\frac{\left(1+\|Dg(z)\|^2\right)^{\frac{m+1+p}2}}{q^m(q-1)^p}
\|z\|^{q(1-p)+q(n-m-1)-m(q-2)}=\frac1{q^m(q-1)^p}.
$$
We conclude the claim \eqref{varphiatu}, and in turn Theorem~\ref{inf0orpos} (ii).
\end{proof}

\section{Almost constant density for $L_p$ surface area measure and the proof of Theorem~\ref{thm1-bounded-density}}
\label{secMain}

The key for the proof of Theorem \ref{thm1-bounded-density} will be the following.

\begin{proposition}\label{prop1-bounded-density}
Let $n$ be a positive integer and $p\in[0,1]$. Then, there exist $\varepsilon_0>0$ and $C>1$, that depend only on $n$ and $p$, such that if $L\in{\cal K}_{(o)}$ has absolutely continuous surface area measure, and
\begin{equation}\label{eq-weaker}
(1-\varepsilon_0)\leq h_L^{1-p}(x)f_L(x)\leq (1+\varepsilon_0),\ \forall x\in\Sn,
\end{equation}
then $h_L\leq C$ and $V(L)>C^{-1}$.
\end{proposition}

To prove Proposition \ref{prop1-bounded-density}, we fix $n$ and $p$.
We will need some linear change of variables formulas. We say that a bounded closed set $S\subset \R^n$ is a
star body
if ${\rm conv}\{o,x\}\subset S$ for $x\in S$, and $o\in{\rm int}\,S$. In this case, the radial function
$\varrho_S(x)=\max\{\lambda\geq 0:\lambda x\in S\}$ is continuous, positive
 and $-1$-homogeneous on $\R^n\setminus\{o\}$
and
$V(S)=\frac{1}{n}\int_{\Sn}\varrho_S(x)^n\, d\Ha(x)$.

\begin{lemma}\label{lemma-equivariance}
Let $\psi:\R^n\setminus\{o\}\to\R_+$ be $-n$-homogeneous and measurable and $T\in GL(n)$. Then, $$\int_{\Sn}\psi(Tx)\, d\Ha(x)=|\det T|^{-1}\int_{\Sn}\psi(x)\, d\Ha(x).$$
\end{lemma}
\begin{proof}
We may assume that $\psi$ is continuous and strictly positive, and hence $\psi=\rho_S^n$ for some star body $S$ and
 $(\psi\circ T)^{1/n}=\rho_{T^{-1}S}$. We conclude that
$$\int_{\Sn}\psi(x)\, d \Ha(x)=nV(S)=n |\det T|V(T^{-1}S)=|\det T|\int_{\Sn}\psi(Tx)\, d\Ha(x).$$
\end{proof}

For a convex body $K\subset\R^n$, $T\in {\rm GL}(n)$
and $x\in\R^n$, we have  $h_{TK}(x)=h_K(T^tx)$. In addition, if $S_K$ is absolute continuous with respect to $\Ha$, then
 $f_K$ can be extended to be 
a $-(n+1)$-homogeneous function on $R^n\setminus\{o\}$, which satisfies  ({\it cf.} \cite{Lut90})
\begin{equation}
\label{fKPhi}
f_{T K}(x)=(\det T)^2f_K(T^*x)\mbox{ \ \ \ for $x\in R^n\setminus\{o\}$}.
\end{equation}

\begin{lemma}\label{lemma-change of variables-1}
If $K\in {\cal K}_o^n$ be such that $S_K$ is absolutely continuous with respect to $\Ha$, $T\in GL(n)$, $p\in\R$
 and $\varphi:\Sn\to\R_+$ is measurable, then
\begin{equation}\label{eq-change of variables-1}
\int_{\Sn} \varphi(x)h_K^{1-p}(x)f_K(x)\, d\Ha(x)=|\det T|^{-1}\int_{\Sn}\varphi\left(\frac{T^tx}{\|T^tx\|}\right)
h_{TK}^{1-p}(x) f_{TK}(x)\|T^tx\|^p\, d\Ha(x).
\end{equation}
\end{lemma}
\begin{proof}
 Since the function 
$$\R^n\setminus\{o\}\ni x\mapsto \varphi(x/\|x\|)h_K^{1-p}(x)f_K(x)\|x\|^p$$is $-n$-homogeneous, Lemma \ref{lemma-equivariance} applied to $T^t$ and \eqref{fKPhi} give
\begin{align*}
 \int_{\Sn}\varphi(x)h_K^{1-p}(x)f_K(x)\, d\Ha(x)&= \int_{\Sn}\varphi\left(\frac{x}{\|x\|}\right)h_K^{1-p}(x)f_K(x)\|x\|^p\, d\Ha(x) \\
 &=|\det T^t| \int_{\Sn}\varphi\left(\frac{T^tx}{\|T^tx\|}\right)
h_{K}^{1-p}(T^tx)f_{K}(T^tx)\|T^tx\|^p\, d\Ha(x)\\
 &=|\det T|^{-1}\int_{\Sn}\varphi\left(\frac{T^tx}{\|T^tx\|}\right)
h_{TK}^{1-p}(x) f_{TK}(x)\|T^tx\|^p\, d\Ha(x).
\end{align*}
\end{proof}

Applying Lemma~\ref{lemma-change of variables-1} to the characteristic functions of Borel sets
in Lemma~\ref{lemma-change of variables-2} leads to the 
 following statement.

\begin{lemma}\label{lemma-change of variables-2}
If $K\in {\cal K}_o^n$ be such that $S_K$ is absolutely continuous with respect to $\Ha$, $T\in GL(n)$, $p\in\R$
and $U\subset\Sn$ is measurable, then
$$
\int_{U} h_K^{1-p}(x)f_K(x)\, d\Ha(x)=|\det T|^{-1}\int_{\{T^{-t}u/\|T^{-t}u\|:u\in U\}}h_{TK}^{1-p}(x) f_{TK}(x)\|T^tx\|^p\, d\Ha(x).
$$
\end{lemma}

\begin{lemma}\label{lemma-change-of-variables-3}
If $T\in GL(n)$, $p\in\R$
 and $\varphi:\Sn\to\R_+$ is measurable, then
$$\int_{\Sn}\varphi(T^tx/\|T^tx\|)\|T^tx\|^p\, d\Ha(x)=
|\det T|^{-1}\int_{\Sn}\varphi(x)\|T^{-t}x\|^{-n-p}\, d\Ha(x).$$
\end{lemma}
\begin{proof}
As the function
$$\R^n\setminus\{o\}\ni x\mapsto \varphi(x/\|x\|)\|T^{-t}x\|^{-(n+p)}\|x\|^p$$is $-n$-homogeneous, the claim follows immediately from Lemma \ref{lemma-equivariance}  applied to the linear transform $T^t$.
\end{proof}

\begin{lemma}\label{lemma-concentration}
Let  $\{\lambda_1^m\}_{m=1}^\infty,\dots,\{\lambda_n^m\}_{m=1}^\infty$ be positive sequences,
that either converge or tend to infinity, 
such that $\lambda_1^m\geq\dots\geq \lambda_n^m$, for all $m$ and such that the sequence $\{\lambda_{j+1}^m/\lambda_j^m\}_{m=1}^\infty$ converges, for $j=1,\dots,n-1$. 
Set $$k:=\min\Big\{j\in\{1,\dots,n\}:\lim_m\frac{\lambda_n^m}{\lambda_j^m}>0\Big\}$$
and
$$T_m:=\textnormal{diag}(1/\lambda_1^m,\dots,1/\lambda_n^m),\qquad m\in\mathbb{N}.$$
If $p\in[0,1)$, and $\{L_m\}$ is a sequence of convex bodies that contain the origin, satisfy
\begin{equation}\label{eq-lambda-again}\lambda^{-1}\leq h_{L_m}^{1-p}f_{L_m}\leq \lambda,\qquad \textnormal{on }\Sn,
\end{equation}
for some $\lambda>0$ and for all $m$ and $\{K_m\}$ is a sequence of convex bodies satisfying
$$L_m=T_mK_m,\qquad m\in\mathbb{N},$$
such that $\{K_m\}$ converges to some convex body $K_\infty$, then $S_{p,K_\infty}$ is concentrated in the subspace
$$H:=\textnormal{span}\{e_k,e_{k+1}\dots,e_n\}.$$
\end{lemma}



Before we prove Lemma~\ref{lemma-concentration}, we need to fix some notations.  We write $A\approx B$ if the quantity $A/B$ is bounded from above and below by positive constants that depend only on $n$ and $ p$. We also write $A\approx_{a_1,\dots,a_l} B$ to denote that the ratio $A/B$ is bounded from above and below by positive constants that depend only on $n$, $ p$ and the parameters $a_1,\dots,a_l\in\R$.

\begin{lemma}\label{lemma-max|T_m^tx|^p}
Let $T_m$, $K_m$ be defined as in Lemma~\ref{lemma-concentration}. If, in addition, $a^{-1}B^n+x_m\subset K_m\subset aB^n$, for some constant $a\geq 1$, for some $x_m\in\R^n$ and for all $m$, then
$$\det T_m\approx_{a,\lambda} \max_{x\in\Sn}\|T_mx\|^p.$$
\end{lemma}
\begin{proof}
Using the fact that $|K_m|\approx_a 1$ and \eqref{eq-lambda-again}, we get 
\begin{eqnarray*}
\det T_m\approx_a V(T_m K_m)&=&\frac{1}{n}\int_{\Sn}h_{T_mK_m}(x)f_{T_mK_m}(x)\, d\Ha(x)\\
&\approx_{a,\lambda}& \int_{\Sn} h_{T_mK_m}^p(x) \, d\Ha(x) =\int_{\Sn} h_{K_m}^p(T_mx)\, d\Ha(x).  
\end{eqnarray*}Thus, the assertion is trivial if $p=0$. Let $0<p<1$. Since
$$\int_{\Sn}\|T_mx\|^p\, d\Ha(x)\approx \max_{x\in \Sn}\|T_mx\|^p,$$ it suffices to prove that 
$$\int_{\Sn}h_{K_m}^p(T_mx)\, d\Ha(x)\approx_a \int_{\Sn}\|T_mx\|^p\, d\Ha(x).$$
The ``$\leq$" part follows immediately from the fact that $h_{K_m}\leq a\|\cdot\|$. For the reverse inequality, simply notice that 
$$\max\{h_{K_m}(x),h_{K_m}(-x)\}\geq a^{-1},\qquad \forall x\in\Sn.$$In particular,
$$h_{K_m}^p(x)+h_{K_m}^p(-x)\geq a^{-p}\|x\|^p,\qquad \forall x\in\R^n$$
and hence 
\begin{align*}
\int_{\Sn}h_{K_m}^p(T_mx)\, d\Ha(x)&=\frac12\int_{\Sn}\left(h_{K_m}^p(T_mx)+h_{K_m}^p(-T_mx)\right)\, d\Ha(x)\\
&\geq \frac{a^{-p}}2\int_{\Sn}\|T_mx\|^p\, d\Ha(x).
\end{align*}
\end{proof}
\begin{proof}[Proof of Lemma \ref{lemma-concentration}]
Since $\{K_m\}$ converges to a convex body, there exists $a>0$ and a sequence of points $\{x_m\}$ in $\R^n$, such that $a^{-1}B^n+x_m\subset K_m\subset aB^n$, for all $m$. Let $U\subset \Sn$ be an open 
such that $H\cap {\rm cl}\,U=\emptyset$.  
As Lemma \ref{lemma-max|T_m^tx|^p} yields the existence 
of a constant $C_1>0$  depending only on $n,\ p$, such that $(\det T_m)^{-1}\|T_m x\|^p\leq C_1$ for $x\in\Sn$,
we deduce from Lemma \ref{lemma-change of variables-2} and \eqref{eq-lambda-again} that
\begin{align}
\nonumber
 \int_Uh^{1-p}_{K_m}(x)f_{K_m}(x)\, d\Ha(x)=&(\det T_m)^{-1}\int_{\{T^{-1}_mu/\|T^{-1}_mu\|:u\in U\}}h^{1-p}_{T_mK_m}(x)f_{T_mK_m}(x)\|T_m x\|^p\, d\Ha(x) \\
\nonumber
 \leq& C_1 \int_{\{T^{-1}_mu/\|T^{-1}_mu\|:u\in U\}}h^{1-p}_{T_mK_m}(x)f_{T_mK_m}(x)\, d\Ha(x)\\
\label{SKmUuppBound}
 \leq&\lambda C_1 \Ha\left(\{T^{-1}_mu/\|T^{-1}_mu\|:u\in U\}\right).
\end{align}
 Since $H\cap {\rm cl}\,U=\emptyset$, there exists $c>0$ such that 
$$\max\{\|u_1\|,\|u_2\|,\dots,\|u_{k-1}\|\}\geq c, \mbox{ \ for }u=(u_1,\dots,u_n)=u_1e_1+\dots+u_ne_n\in U.$$
It follows that if $u\in U$ and $i\in\{k,k+1,\dots,n\}$,  then
$$
\left|\left\langle\frac{T_m^{-1}u}{\|T_m^{-1}u\|},e_i\right\rangle\right|=\frac{\lambda_i^m|u_i|}{\sqrt{(\lambda_1^mu_1)^2+\dots+(\lambda_n^mu_n)^2}}
\leq\frac{\lambda_i^m|u_i|}{c\lambda_{k-1}^m}\leq \frac{1}{c}\frac{\lambda_k^m}{\lambda_{k-1}^m}.
$$
We note that $\lambda^m_k/\lambda_{k-1}^m\to 0$, and hence for any $\varepsilon>0$, there exists $m_0\in\mathbb{N}$, such that if $m\geq m_0$, then
$$\left|\left\langle\frac{T_m^{-1}u}{\|T_m^{-1}u\|},e_i\right\rangle\right|<\varepsilon.$$
It follows that if  $\varepsilon>0$, and $m$ is large enough, then the set $\{T^{-1}_mu/\|T^{-1}_mu\|:u\in U\}$ is contained in $(\Sn\cap H)_\varepsilon:=\{x\in\Sn:d(\Sn\cap H,x)<\varepsilon\}$, and hence 
$$\Ha\left(\{T^{-1}_mu/\|T^{-1}_mu\|:u\in U\}\right)\to 0 \mbox{ \ as $m\to \infty$. }$$
Since $U$ is open and since $\{S_{p,K_m}\}$ converges weakly to $S_{p,K_\infty}$, we deduce
from \eqref{SKmUuppBound} that
$$S_{p,K_\infty}(U)\leq \liminf_mS_{p,K_m}(U)= 0.$$
Therefore, ${\rm supp}\,S_{p,K_\infty}\subset H$.
\end{proof}

Lemma~\ref{structure} yields  the following:
\begin{lemma}\label{lemma-also-concentrated}
If $K\in{\cal K}_o^n$ is such that $S_{p,K}$ is concentrated in $H:=\textnormal{span}\{e_k,e_{k+1},\dots,e_n\}$, where $k\in\{2,\dots,n\}$, and $T$ is a diagonal and positive definite operator, then $S_{p,TK}$ is also concentrated in $H$.    
\end{lemma}

The last auxiliary statement we need to verify  Proposition~\ref{prop1-bounded-density} is the following one:

\begin{lemma}
\label{volumeLower}
For $n\geq 2$, $p\in[0,1)$ and $\theta>1$,  there exists a constant $c=(\theta,n,p)>0$ such that if $L\in {\cal K}^n_{(o)}$ has absolutely continuous surface area measure and satisfy 
$$
 h_L^{1-p} f_L\geq \theta^{-1} \mbox{ \ and  $h_L\leq \theta$ on }\Sn,
$$
then $V(L)\geq c$.
\end{lemma}
\begin{proof}
We may assume that $p\in(0,1)$. Writing $S(L)$ to denote the surface area of $L$, it follows form the H\"older inequality that
\begin{align*}
\theta^{-1}\,\Ha(\Sn)&\leq \int_{\Sn}h^{1-p}_{L}(x)f_{L}(x)\,d\Ha(x)\\
&\leq
\left(\int_{\Sn}h_{L}(x)f_{L}(x)d\Ha(x)\right)^{1-p}\left(\int_{\Sn}1\cdot f_{L}(x)\,d\Ha(x)\right)^p\\
&=
(nV(L))^{1-p}S(L)^p\leq(nV(L))^{1-p}\theta^{p(n-1)}\Ha(\Sn)^p.
\end{align*}
\end{proof}

\begin{proof}[Proof of Proposition~\ref{prop1-bounded-density}]
The proof is by contradiction. We suppose that the assertion of the Proposition is not true, and hence,
using also Lemma~\ref{volumeLower}, 
 for any positive integer $m$, there exists a convex body $L_m\in {\cal K}_{(o)}$, such that 
\begin{equation}\label{eq-thm1-bounded-density-1}
1-\frac{1}{m}\leq h^{1-p}_{L_m}(x)f_{L_m}(x)\leq 1+\frac{1}{m},\qquad\forall x\in\Sn    
\end{equation}
and
\begin{equation}\label{eq-thm1-bounded-density-2}
\sup_{x\in\Sn}h_{L_m}(x)>m.  
\end{equation}
Next, note that for each $m$, there exists $T_m\in GL(n)$, such that the maximal volume ellipsoid contained in $K_m:=T_m^{-1}L_m$ is $x_m+B^n$ for $x_m\in\R^n$. Then  
$K_m\subset x_m+nB^n$ by the classical theorem of F. John, and hence $o\in K_m$ yields
\begin{equation}\label{eq-thm1-bounded-density-3}
x_m+B^n\subset K_m\subset 2nB^n.  
\end{equation}

Since \eqref{eq-thm1-bounded-density-1}, \eqref{eq-thm1-bounded-density-2} and \eqref{eq-thm1-bounded-density-3} are invariant under composition with orthogonal maps, we may assume that 
$$T_m^{-1}=\textnormal{diag}(\lambda_1^m,\dots,\lambda^m_n),$$
for some $\lambda_1^m\geq \dots\lambda^m_n>0$. Furthermore, by taking subsequences, we may assume that all sequences $\{\lambda_j^m\}$ either converge or tend to infinity, all sequences $\{\lambda_{j+1}^m/\lambda_j^m\}$ converge and that $K_m\to K_\infty$, for some convex body $K_\infty$.

{\bf Claim 1.} $\lambda_1^m\to\infty$ and $\lambda_n^m\to 0$.
\begin{proof}
Set $s_j^m:=(\lambda_j^n)^{-1}$, $j=1,\dots,n$. As $s_1^m$ and $s_n^m$ are the smallest and the largest eigenvalue of $T_m$ respectively, it is equivalent to prove that $s_n^m\to \infty$ and $s_1^m\to 0$. 

Here $s_n^m\to \infty$ follows from the fact that the sequence $\{K_m\}$ is bounded, while the sequence $\{T_mK_m\}$ is unbounded. For $s_1^m\to 0$, notice that Lemma \ref{lemma-max|T_m^tx|^p} implies
$$s_1^m\dots s_n^m\approx \max_{x\in \Sn}\|T_mx\|^p=(s_n^m)^p,$$
and hence  $$s_1^ms_2^m\dots (s_n^m)^{1-p}\approx 1.$$
Therefore, $s_1^m\to 0$  follows from that fact that $\lim_m (s_n^m)^{1-p}=\infty$. 
\end{proof}

To continue with the proof of Proposition~\ref{prop1-bounded-density}, as in Lemma \ref{lemma-concentration}, set 
$$k:=\min\Big\{j\in\{1,\dots,n\}:\lim_m\frac{\lambda_n^m}{\lambda_j^m}>0\Big\}\qquad\textnormal{and}\qquad H:=\textnormal{span}\{e_k,\dots,e_n\}.$$
It follows by Claim 1 that $k>1$, hence $\dim H\leq n-1$. Moreover, Lemma \ref{lemma-concentration} shows that $S_{p,K_\infty}$ is concentrated in $H$.

For $m\in\mathbb{N}$, set $$S_m:=\textnormal{diag}\left(1,\dots,1,\frac{\lambda^m_n}{\lambda^m_k},\dots,\frac{\lambda^m_n}{\lambda^m_{n-1}},1\right)$$
and 
$$K'_m:=S_mK_m=S_mT_m^{-1}L_m.$$
Thus, $L_m=\Phi_mK'_m$, where
$$\Phi_m=\textnormal{diag}\left(\frac{1}{\lambda_1^n},\dots,\frac{1}{\lambda_{k-1}^m},\frac{1}{\lambda^m_n},\dots,\frac{1}{\lambda^m_n}\right).$$
Observe that $S_m\to S$, for some diagonal positive definite matrix $S$. Hence, $K'_m\to SK_\infty$. Since $S_{p,K_\infty}$ is concentrated in $H$ and $S$ is diagonal, Lemma \ref{lemma-also-concentrated} shows that $S_{p,SK_\infty}$ is also concentrated in $H$.

On the other hand, Lemma \ref{lemma-change of variables-1} shows that for any continuous function $\varphi:\Sn\to \R_+$, we have
\begin{eqnarray*}
\int_{\Sn}\varphi(x)h_{K'_m}^{1-p}(x)f_{K'_m}\, d\Ha(x)=\det\Phi_m^{-1}\int_{\Sn}\varphi(\Phi_mx/\|\Phi_mx\|)h^{1-p}_{L_m}(x)f_{L_m}(x)\|\Phi_mx\|^p\, d\Ha(x) ,
\end{eqnarray*}
which, by \eqref{eq-thm1-bounded-density-1}, gives
$$1-\frac{1}{m}\leq\frac{\int_{\Sn}\varphi(x)h_{K'_m}^{1-p}(x)f_{K'_m}(x)\, d\Ha(x)}{\det \Phi_m^{-1}\int_{\Sn}\varphi(\Phi_mx/\|\Phi_mx\|)\|\Phi_mx\|^p\, d\Ha(x)}\leq 1+\frac{1}{m}.$$
Using Lemma \ref{lemma-change-of-variables-3}, we conclude that for each $m$, it holds
$$1-\frac{1}{m}\leq\frac{\int_{\Sn}\varphi(x)h_{K'_m}^{1-p}(x)f_{K'_m}(x)\, d\Ha(x)}{\det \Phi_m^{-2}\int_{\Sn}\varphi(x)\|\Phi_m^{-1} x\|^{-n-p} \, d\Ha(x)}\leq 1+\frac{1}{m}.$$
According to Schneider \cite{Sch14}, the restriction of $h_{K'_m}$ to $\Sn$ tends uniformly to 
$h_{SK_\infty}$, and the surface area measure $S_{K'_m}$  tends weakly to 
$S_{SK_\infty}$, and hence  $dS_{p,K'_m}=h_{K'_m}^{1-p}f_{K'_m}\, d\Ha$  tends weakly to 
$S_{p,SK_\infty}$. We deduce that
\begin{eqnarray}\label{eq-thm1-bounded-density-4}
\lim_m\det\Phi_m^{-2}\int_{\Sn}\varphi(x)\|\Phi_m^{-1}x\|^{-n-p}d\, \Ha(x)&=&\lim_m\int_{\Sn}\varphi(x)h_{K'_m}^{1-p}(x)f_{K'_m}(x)\, d\Ha(x)\nonumber\\
&=&\int_{\Sn}\varphi(x)\, dS_{p,SK_\infty}(x).
\end{eqnarray}

Let $O$ be any orthogonal map that fixes the elements of $H^\perp$. Then, according to \eqref{eq-thm1-bounded-density-4}, we have
\begin{eqnarray*}
 \int_{\Sn}(\varphi\circ O^t)(x)\, dS_{p,SK_\infty}(x)&=&\lim_m\det\Phi_m^{-2}\int_{\Sn}\varphi(O^tx)\|\Phi_m^{-1}x\|^{-n-p}\,d\Ha(x) \\
&=&\lim_m\det\Phi_m^{-2}\int_{\Sn}\varphi(y)\|\Phi_m^{-1}Oy\|^{-n-p}\, d\Ha(y)\\
&=&\lim_m\det\Phi_m^{-2}\int_{\Sn}\varphi(y)\|\Phi_m^{-1}y\|^{-n-p}\, d\Ha(y)\\
 &=&\int_{\Sn}\varphi(x)\, dS_{p,SK_\infty}(x).
\end{eqnarray*}
Consequently, $S_{p,SK_\infty}(OA)=S_{p,SK_\infty}(A)$, for any such $O$ and for any Borel set $A$ in $\Sn$. Therefore, since $S_{p,SK_\infty}$ is concentrated in $H$, it follows that the restriction of $S_{p,SK_\infty}$ to the class of Borel subsets of $H$ is a Haar measure on $\Sn\cap H$, and hence it is a multiple of the spherical Lebesgue measure on $\Sn\cap H$. However (recall that $\dim H\leq n-1$), by \cite[Theorem 1.3]{Sar}, this is a contradiction, completing the  proof of Proposition~\ref{prop1-bounded-density}. 
\end{proof}
\begin{lemma}\label{lemma-o in int K}
There exist constants $\varepsilon_1>0$ and $c>0$, that depend only on $n$ and $p$, such that if $\|h_K^{1-p}f_K-1\|_{L^\infty}<\varepsilon_1$, then 
$h_K(x)\geq c$, for all $x\in \Sn$.
\end{lemma}

\begin{proof}
If $h_K^{1-p}f_K=1$, everywhere on $\Sn$, then $K=B^n$ (even if we do not assume that $o\in {\rm int}\, K$; see \cite[Remark 3.5]{Sar}). If the assertion of the lemma does not hold, then we can find a sequence of convex bodies $\{K_m\}$, such that $\|h_{K_m}^{1-p}f_{K_m}-1\|_{L^\infty}\to 0$ and $\inf_{x\in\Sn}h_{K_m}(x)\to 0$. However, by Proposition \ref{prop1-bounded-density}, the sequence $\{h_{K_m}\}$ is bounded, and $V(K_m)$ is bounded away from zero, and hence it has a subsequence that converges to some convex body $K'$. But then, $h_{K'}^{1-p}f_{K'}\equiv 1$ and $\inf_{x\in\Sn}h_{K'}(x)=0$. This is a contradiction, completing our proof.
\end{proof}
The final part of the proof of Theorem~\ref{thm1-bounded-density} follows the argument in \cite{CHLL20}, \cite{CFL22}. However, in our opinion, it is slightly less technical.
\begin{proof}[Proof of Theorem~\ref{thm1-bounded-density}]
Consider the map
$$F:H^2(\Sn)\ni u\mapsto u^{1-p}\det(u_{ij}+\delta_{ij}u)_{i,j=1}^{n-1}\in L^2(\Sn).$$
Then, $F$ is a smooth map between Banach spaces and its Frechet derivative $T:H^2(\Sn)\to L^2(\Sn)$ at $v\equiv 1$ is given by
$$T(u)=\Delta_{\Sn}u+(n-p)u.$$
Notice that $T$ is bounded and invertible, since the eigenvalues of $-\Delta_{\Sn}$ are exactly the positive integers of the form $k^2+(n-2)k$, $k=0,1,\dots$. Thus, by the inverse function theorem for Banach spaces, there exists an $H^2$-neighbourhood
of the constant function 1, such that for $u_1,u_2$ in this neighbourhood, $F(u_1)=F(u_2)$ implies $u_1=u_2$.
In particular, there exists a $C^2$-neighbourhood $N$ of 1, such that if $h_K, h_L\in N$, and $S_{p,K}=S_{p,L}$, then $K=L$. Thus, it suffices to prove that there exists $\varepsilon>0$, such that if $\|g-1\|_{C^a}<\varepsilon$ and $dS_{p,K}=gdx$, then $h_K\in N$.

Assume not. Then, there exists a sequence $\{K_m\}$ of convex bodies, such that
\begin{equation}\label{eq-application-1}
\|h_{K_m}^{1-p} f_{K_m}-1\|_{C^a}\to 0,\qquad\textnormal{as }m\to\infty,   
\end{equation}
but $h_{K_m}\not\in N$, for all $m$. The latter can be written as
\begin{equation}\label{eq-application-2}
 \|h_{K_m}-1\|_{C^2}>\eta_0,   
\end{equation}
for some fixed $\eta_0>0$.

It follows from Proposition~\ref{prop1-bounded-density} that $\{K_m\}$ is a bounded sequence. 
Thus, we may assume that $\{K_m\}$ converges to some compact convex set $K$. Notice that $K$ has to be full dimensional, otherwise the sequence $\{h_{K_m}^{1-p}f_{K_m}\}$ would converge weakly to 0, which is impossible by \eqref{eq-application-1}.

Notice that $S_{p,{K_m}}$ converges weakly to $S_{p,K}$. However, $h_{K_m}^{1-p}f_{K_m}$ converges to 1 uniformly. Consequently, $S_{p,K}={\cal H}^{n-1}$. By Lemma \ref{lemma-o in int K}, we conclude that $K=B^n$. That is, $\|h_{K_m}-1\|_\infty\to 0$, as $m\to\infty$.

Next, set $g_m:=h_{K_m}^{1-p}f_{K_m}$ and notice that the Lipschitz constant of $h_{K_m}$ is bounded from above by the outradius of $K_m$. Thus, $\|h_{K_m}\|_{C^1}$ is uniformly bounded, which shows (since $\|g_m\|_{C^a}$ is uniformly bounded and since, by Lemma \ref{lemma-o in int K}, $h_{K_m}$ is bounded away from zero) that $\|f_{K_m}\|_{C^a}=\|h_{K_m}^{p-1}g_m\|_{C^a}$ is uniformly bounded. By Caffarelli's regularity theory (see De Philippis, Figalli \cite{DPF14}), there is a constant $C_0>0$ (which does not depend on $m$), such that
$$\|h_{K_m}\|_{C^{2,a}}\leq C_0.$$
By the Arzela-Ascoli Theorem, $\{h_{K_m}\}$ has a subsequence $\{h_{K_{l_m}}\}$ that converges in $C^2$. Since $\|h_{K_{l_m}}-1\|_\infty\to 0$, we conclude that
$$\|h_{K_{l_m}}-1\|_{C^2}\to 0.$$
This contradicts \eqref{eq-application-2} and our proof is complete.
\end{proof}


\section{Bounded density for $L_p$ surface area measure 
and the proof of Theorem~\ref{thm2-bounded-density}}
\label{secBoundedDensity}

Let us recall Problem~\ref{problem 1} from Chen, Feng, Liu \cite{CFL22}:

\begin{problem}[Chen, Feng, Liu \cite{CFL22}]\label{problem-10}
For $n\geq 2$, $p\in[0,1)$ and $\lambda>1$, does there exists a $C=C(\lambda,n,p)>0$ such that if $L\in {\cal K}^n_{(o)}$ with absolutely continuous surface area measure, satisfying 
\begin{equation}
\label{eq-main}
\lambda^{-1}\leq h_L^{1-p} f_L\leq \lambda,    \qquad \textnormal{on }\Sn,
\end{equation}
then $h_L\leq C$ and $V(L)\geq C^{-1}$.
\end{problem}

Theorem \ref{thm2-bounded-density} will follow immediately from Theorem \ref{inf0orpos} and the following result, which will be proved in the end of this section.

\begin{proposition}\label{prop-bounded-density-below} If  for some $n\geq 2$ and for some $p\in[0,1)$, the answer to Problem \ref{problem-10} is negative, then, there exists $K\in{\cal K}_o^n$, with $o\in\partial K$, such that $S_{p,K}$ is concentrated in a proper subspace $H$ of $\R^n$ and the restriction of $S_{p,K}$ to the class of Borel subsets of $\Sn\cap H$ is absolutely continuous with respect to the spherical Lebesgue measure on $\Sn\cap H$ and its density is bounded and bounded away from zero.    
\end{proposition}

Before proving 
Proposition~\ref{prop-bounded-density-below}, we need a fact from elementary measure theory.

\begin{lemma}\label{lemma-measure}
If $\mu$ is a finite non-trivial Borel measure on $\Sn$ such that there exists a constant $A>0$, with the property that 
$$\mu(V_1)\leq A\mu({\rm cl}\,V_2)$$
for any two open spherical caps $V_1$, $V_2$ of the same diameter, then, $\mu$ is absolutely continuous with respect to $\Ha$ such that its density function can be chosen to be bounded and bounded away from zero.
\end{lemma}
\begin{proof}
For $x\in \Sn$, $\varepsilon>0$, set $U(x,\varepsilon):=\{y\in\Sn:\|x-y\|<\varepsilon\}$. It is well known that there exists a constant $C_0>0$ , that depends only on $n$ (actually $C_0=2^{n-1}$ works), such that 
\begin{equation}
\label{doubling}
\Ha(U(x,2\varepsilon))\leq C_0 \Ha(x,\varepsilon).
\end{equation}
For $\varepsilon>0$, fix a maximal set $\{x_1,\dots,x_N\}\subset \Sn$, with the
property that the spherical caps $U(x_1,2\varepsilon),\dots,U(x_N,2\varepsilon)$ are disjoint, and hence
$$\bigcup_{j=1}^NU(x_j,4\varepsilon)=\Sn.$$
Set $V_j:=U(x_j,\varepsilon)$, $j=1,\dots,N$. Then, ${\rm cl}\,V_1,\dots,{\rm cl}\,V_N$ are  disjoint, and we have
\begin{eqnarray*}
\sum_{j=1}^N\mu({\rm cl}\,V_j)\leq \mu(\Sn)&=&\frac{\mu(\Sn)}{\Ha(\Sn)}\Ha(\Sn) \\
&\leq& \frac{\mu(\Sn)}{\Ha(\Sn)}\sum_{j=1}^N\Ha(U(x_j,4\varepsilon))\\
&\leq&C_0^2\frac{\mu(\Sn)}{\Ha(\Sn)}\sum_{j=1}^N\Ha(V_j)=B\sum_{j=1}^N\Ha(V_j)
\end{eqnarray*}
for $B=C_0^2\frac{\mu(\Sn)}{\Ha(\Sn)}$.
Therefore, there exists $t\in\{1,\dots,N\}$, such that $$\mu({\rm cl}\,V_t)\leq B\Ha(V_t).$$
It follows by our assumption that for any open spherical cap $V$ of $\Sn$, of radius $\varepsilon$, it holds $$\mu(V)\leq A\mu({\rm cl}\,V_t)\leq A\cdot B\,\Ha(V).$$
Since $\varepsilon$ is arbitrary, it follows from 
Vitali's Covering Lemma, the regularity of Hausdorff measure and the doubling property \eqref{doubling}
  that $\mu$ is absolutely continuous with respect to $\Ha$ having bounded density.

Next, we suppose that the essential infimum of the density of $\mu$ is  zero, and seek a contradiction. Then, for each $\eta>0$, there exists a point $x_\eta\in\Sn$, such that the limit $\lim_{\delta\to 0^+}\mu(U(x_\eta,\delta)/\Ha(U(x_\eta,\delta))$ exists and is less than $\eta/(2A)$. It follows that if $\delta $ is sufficiently small, then 
$$\mu(U(x_\eta,\delta)<\frac{\eta}{A}\cdot \Ha(U(x_\eta,\delta)).$$
Consequently, using again the assumption of Lemma~\ref{lemma-measure} and the absolute continuity of $\mu$ (established previously), we conclude that for any $y\in\Sn$, it holds
$$\limsup_{\delta\to 0^+}\frac{\mu(U(y,\delta))}{\Ha(U(y,\delta))}<\eta.$$  Since $\eta $ is arbitrary, we arrive at
$$\limsup_{\delta\to 0^+}\frac{\mu(U(y,\delta))}{\Ha(U(y,\delta))}=0,\qquad \forall y\in\Sn.$$
As $\mu$ is absolute continuous, we deduce that $\mu\equiv 0$, which is a contradiction. This completes the proof of the Lemma~~\ref{lemma-measure}.
\end{proof}

\begin{proof}[Proof of Propostion \ref{prop-bounded-density-below}]
We suppose that the assertion of Proposition~\ref{prop-bounded-density-below} does not hold, and seek a contradiction. Then,
using Lemma~\ref{volumeLower}, there exists a sequence $L_m$ of convex bodies that contain the origin in their interior and whose surface area measure is absolutely continuous with respect to $\Ha$, and satisfy \eqref{eq-main} and
$$
\lim_{m\to\infty}\sup_{x\in\Sn} h_{L_m}(x)= \infty.
$$
For each $m$, there exists $T_m\in GL(n)$ and a convex body $K_m$ whose maximal volume ellipsoid is a translate of $B^n$, such that $L_m=T_mK_m$. As in the proof of Theorem \ref{thm1-bounded-density}, we may assume that $\{K_m\}$ converges to some convex body $K_\infty$ and that $T_m=\textnormal{diag}{(\lambda_1^m)^{-1},\dots,(\lambda_n^m)}^{-1})$, where the numbers $\lambda_j^m$ satisfy the assumptions of Lemma \ref{lemma-concentration}. Then, by Lemma \ref{lemma-concentration}, $S_{p,K_\infty}$ is concentrated in 
$H=\textnormal{span}\{e_k,e_{k+1},\dots,e_n\}$,  where
$$k=\min\Big\{j\in\{1,\dots,n\}:\lim_m\frac{\lambda_n^m}{\lambda_j^m}>0\Big\}.$$
Therefore, by Lemma \ref{lemma-measure}, in order to prove Proposition \ref{prop-bounded-density-below} (with $K=K_\infty$), it suffices to show that there is a constant $A>0$, such that if $V_1$, $V_2$ are open spherical caps in $\Sn\cap H$, of the same diameter, then
$$S_{p,K_\infty}(V_1)\leq AS_{p,K_\infty}({\rm cl}\,V_2).$$
Equivalently, since $S_{p,K_\infty}$ is concentrated in $H$, it suffices to prove that if $U_1,U_2$ are open spherical caps in $\Sn$, of the same diameter, whose centers are contained in $H$, then
$$S_{p,K_\infty}(U_1)\leq AS_{p,K_\infty}({\rm cl}\,U_2).$$
Notice that, since $U_1$ is open and ${\rm cl}\,U_2$ is closed, it holds
$$S_{p,K_\infty}(U_1)\leq\liminf_mS_{p,K_m}(U_1)\qquad\textnormal{and}\qquad S_{p,K_\infty}({\rm cl}\,U_2)\geq \limsup_m S_{p,K_m}({\rm cl}\,U_2)=\limsup_mS_{p,K_m}(U_2).$$
Therefore, it suffices to prove that there is a constant $A>0$, such that  for all $m$ and for all spherical caps $U_1$, $U_2$ in $\Sn$ of the same diameter, whose centers are contained in $H$, it holds
\begin{equation}\label{eq-bounded-density-desired}S_{p,K_m}(U_1)\leq A S_{p,K_m}(U_2).\end{equation}
Let $U_1$, $U_2$ be as above. Using Lemma \ref{lemma-change of variables-1}, \eqref{eq-main} and Lemma \ref{lemma-change-of-variables-3}, we see that
\begin{eqnarray}\label{eq-last-bounded-density}
S_{p,K_m}(U_i)&=&\int_{U_i} h_{K_m}^{1-p}(x)f_{K_m}(x)\, d\Ha(x)\nonumber\\
&=&\int_{\Sn}{\bf 1}_{U_i}(x)h_{K_m}^{1-p}(x)f_{K_m}(x)\, d\Ha(x)\nonumber\\
&=&\det T_m^{-1}\int_{\Sn}{\bf 1}_{U_i}(T_mx/\|T_mx\|)h_{T_mK_m}^{1-p}(x)f_{T_mK_m}(x)\|T_mx\|^p\, d\Ha(x)\nonumber\\
&\approx_\lambda&\det T_m^{-1}\int_{\Sn}{\bf 1}_{U_i}(T_mx/\|T_mx\|)\|T_mx\|^p\, d\Ha(x)\nonumber\\
&=&\det T_m^{-2}\int_{\Sn}{\bf 1}_{U_i}(x)\|T_m^{-1}x\|^{-n-p}\, d\Ha(x)\nonumber\\
&=&\det T_m^{-2}\int_{U_i}\|T_m^{-1}x\|^{-n-p}\, d\Ha(x),\ \ i=1,2.
\end{eqnarray}
Notice that there exists $O\in O(n)$, that fixes the elements of $H^\perp$, such that $U_2=OU_1$.  Moreover, since $\lim_m\lambda_n^m/\lambda_k^n>0$, there exists $0<\mu\leq 1$, such that 
$$\lambda_n^m\geq \mu\lambda_k^m\geq \mu \lambda_{k+1}^m\geq\dots\geq \mu\lambda_n^m,\qquad \forall m\in\mathbb{N}.$$ Thus, for $x=(x_1,\dots,x_n)=x_1e_1+\dots +x_ne_n\in\Sn$, it holds
\begin{eqnarray*}
\|T_m^{-1}Ox\|^2&=&\|(\lambda_1^mx_1,\dots,\lambda_{k-1}^mx_{k-1})\|^2+\|O(0,\dots,0,\lambda_k^mx_k+\dots+\lambda_{n}^mx_{n})\|^2\\
&\geq&\|(\lambda_1^mx_1,\dots,\lambda_{k-1}^mx_{k-1})\|^2+\|O(0,\dots,0,\lambda_n^mx_k+\dots+\lambda_{n}^mx_n)\|^2\\
&=&(\lambda_1^mx_1)^2+\dots+(\lambda_{k-1}x_{k-1})^2+(\lambda_n^mx_k)^2+\dots+(\lambda_n^mx_n)^2\\
&\geq&\mu^2((\lambda_1^mx_1)^2+\dots+(\lambda_{n}x_{n})^2)=\mu^2\|T_m^{-1}x\|^2,\qquad\forall m\in\mathbb{N}.
\end{eqnarray*}
It follows that for each positive integer $m$, we have
\begin{eqnarray*}
\int_{U_2}\|T_m^{-1}x\|^{-n-p}\, d\Ha(x)&=&\int_{OU_1}\|T_m^{-1}x\|^{-n-p}\, d\Ha(x)\\
&=&\int_{U_1}\|T_m^{-1}Ox\|^{-n-p}\, d\Ha(x)\leq \mu^{-n-p}\int_{U_1}\|T_m^{-1}x\|^{-n-p}\, d\Ha(x).\end{eqnarray*}
This, together with \eqref{eq-last-bounded-density} yields \eqref{eq-bounded-density-desired} (with the constant $A$ depending on $\lambda$ and $\mu$), ending our proof.
\end{proof}

\section{Severe non-uniqueness for  $-n<p<0$ and the proof of Theorem~\ref{thm-non-uniqueness}}
\label{secNonUnique}

For the proof of Theorem~\ref{thm-non-uniqueness}, we will need the following.
\begin{lemma}\label{l-non-uniqueness}
Let $L\in{\cal K}_o^n$, such that $o\in {\rm int}\, L$ and $p<0$. Then, for $x\in L$,
$$\int_{\Sn}h_{L-x}^p(u)\,dS_{p,L}(u)=\inf_{y\in L}\int_{\Sn}h_{L-y}^p(u)\, dS_{p,L}(u),$$if and only if $x=o$.
\end{lemma}
\begin{proof}
Since $h_{L-x}(y)=h_L(y)-\langle y,x\rangle$, the function $L\ni x\mapsto h_{L-x}^p$ is strictly convex, the function 
$$
F:L\ni x\mapsto \int_{\Sn}h_{L-x}^p(u)\, dS_{p,L}(u)
$$ 
is also strictly convex. For $z\in\R^n$, we have
$$
\langle\nabla F(x),z\rangle=-p\int_{\Sn} \langle z,u\rangle h_{L-x}(u)^{p-1}h_L(u)^{1-p}\, dS_{L}(u),
$$
and hence $\nabla F(o)=o$ follows from the fact that the barycentre of $S_L$ is at the origin. We deduce from the strict convexity of $F$ that $x=o$ is the unique minimum point of $F$. 
\end{proof}
\begin{proof}[Proof of Theorem \ref{thm-non-uniqueness}] Define the functional
$$J_{p,\mu}(M):=\inf_{x\in M}\frac{1}{n}\int_{\Sn} h_{M-x}^p(u)\, d\mu(u),\qquad M\in{\cal K}_o^n.$$
If ${\cal K}={\cal K}_s^n$ and $\mu$ is even, it follows easily from Lemma \ref{l-non-uniqueness} that
$$J_{p,\mu}(M)=\frac{1}{n}\int_{\Sn} h_{M}^p(u)\, d\mu(u),\qquad M\in{\cal K}_s^n.$$
If, in addition, $\mu$ is absolutely continuous with respect to ${\cal H}^{n-1}$ with $L^\infty$ density, consider the maximization problem
$$\textnormal{maximize}\ \ J_{p,\mu}(M)V(M)^{-p/n},\qquad\textnormal{subject to }M\in{\cal K}.$$
It is known (see \cite{BBCY19} and \cite{JLW15} ) that if $\mu$ is as above, then the problem indeed admits a maximizer $M_0$ in ${\cal K}$. Furthermore, $M_0$ can be chosen so that 
\begin{equation}\label{eq-non-uniqueness-variational}S_{p,M_0}=\mu.\end{equation}
For $L,M\in{\cal K}$, define also 
$$J_p(M,L):=J_{p,S_{p,M}}(L).$$
Notice that if $o\in {\rm int}\, M$, then it follows that 
\begin{equation}\label{eq-J_p(M,M)=V(M)} 
    J_p(M,M)=V(M).
\end{equation}
This is trivial to verify, if ${\cal K}={\cal K}_s^n$ and follows immediately from Lemma \ref{l-non-uniqueness}, if ${\cal K}={\cal K}_o^n$.

{\bf Claim 1.} Let $K\in {\cal K}_+$, with $o\in{\rm int}\,K$ and $p\in (-n,0)$. If for some $L\in{\cal K}$, it holds
\begin{equation}\label{eq-not} J_p(K,L)>V(K)^{(n-p)/n}V(L)^{p/n},\end{equation}
then $K\in{\cal A}_p$.
\begin{proof}
We show first that if \eqref{eq-not} holds for some $L\in{\cal K}$, then there is a convex body $L_p\in{\cal K}$, such that $L_p\neq K$ and $S_{p,L_p}=S_{p,K}$. Indeed, \eqref{eq-not} and \eqref{eq-J_p(M,M)=V(M)}  show that 
$$\sup_{L\in {\cal K}}J_p(K,L)V(L)^{-p/n}>V(K)^{(n-p)/n}=J_p(K,K)V(K)^{-p/n},$$
while \eqref{eq-J_p(M,M)=V(M)} shows that there exists $L_p\in{\cal K}$, such that $S_{p,L_p}=S_{p,K}$ and $J_p(K,L_p)V(L_p)^{-p/n}=\sup_{L\in {\cal K}}J_p(K,L)V(L)^{-p/n}$. The latter shows that $L_p\neq K$ and our assertion is proved. 

Therefore, to prove our claim, it suffices to show that, given $q\in(-n,p)$, there exists $L'\in {\cal K}$ such that  
\begin{equation}\label{eq-not-q}J_q(K,L')>V(K)^{(n-q)/n}V(L')^{q/n}.\end{equation}
To this end, fix $L\in{\cal K}$, such that $K$ and $L$ satisfy \eqref{eq-not}. Then, there exists $x\in {\rm int}\,L$, such that $\frac{1}{n}\int_{\Sn}h^q_{L-x}(u)\, dS_{q,K}(u)=J_q(K,L)$.
Setting $L':=L-x$ and using H\"older's inequality, we find
\begin{eqnarray*}
V(K)^{(n-p)/n}V(L')^{p/n}&=&  V(K)^{(n-p)/n}V(L)^{p/n}<J_p(K,L)=J_p(K,L')\\
&\leq&\frac{1}{n}\int_{\Sn}h_{L'}^p(u)\, dS_{p,K}(u)=\frac{1}{n}\int_{\Sn}\left(\frac{h_K(u)}{h_{L'}(u)}\right)^{-p}h_K(u)\, dS_K(u)\\
&\leq&\frac{1}{n}\left(\int_{\Sn}\left(\frac{h_K(u)}{h_{L'}(u)}\right)^{-q}h_K(u)\, dS_K\right)^{p/q}\left(\int_{\Sn}h_K(u)\, dS_K\right)^{(q-p)/q}\\
&=&J_q(K,L')^{p/q}V(K)^{(q-p)/q}.
\end{eqnarray*}
Readily, this gives \eqref{eq-not-q} and our claim is proved.
\end{proof}

To finish with the proof of Theorem \ref{thm-non-uniqueness}, for $Q\in{\cal K}$, define  
$$I_p(Q):=\sup_{L\in {\cal K}}J_p(Q,L)V(Q)^{(p-n)/n}V(L)^{-p/n}.$$
It follows by Claim 1 that it suffices to prove that the set of convex bodies $Q\in{\cal K}_+\cap {\cal K}_{(o)}$ that satisfy $I_p(Q)>1$ is dense in ${\cal K}$, in the Hausdorff metric. In fact, since ${\cal K}_+\cap {\cal K}_{(o)}$ is dense in ${\cal K}_{(o)}$, it suffices to prove that the set of convex bodies $Q\in{\cal K}_{(o)}$ that satisfy $I_p(Q)>1$ is dense in ${\cal K}$. To see this, simply notice that any polytope $P\in{\cal K}$ with $o\in{\rm int}\,P$, which has at least one pair of parallel facets (this holds automatically if ${\cal K}={\cal K}_s^n$), satisfies $I_p(P)=\infty>1$. We conclude the proof by the fact that the set of polytopes in ${\cal K}$ that contain the origin in their interior and have at least one pair of parallel facets is dense in ${\cal K}$, in the Hausdorff metric.
\end{proof}

\noindent{\bf Acknowledgement} K\'aroly B\"or\"oczky is supported by NKFI grant K 132002. We are grateful for the extremely useful discussions with Alessio Figalli and Pengfei Guan, and for the comments of the referee.

\noindent K\'aroly J. B\"or\"oczky, HUN-REN Alfr\'ed R\'enyi Institute of Mathematics, Budapest, Hungary\\
boroczky.karoly.j@renyi.hu\\

\noindent Christos Saroglou, Department of Mathematics, University of Ioannina, Greece\\
 csaroglou@uoi.gr

\end{document}